\documentclass[12pt]{amsart}

\usepackage{amssymb, amsmath,mathrsfs}
\usepackage{xargs} 
\usepackage[pagebackref,hypertexnames=false, colorlinks, citecolor=red, linkcolor=red]{hyperref}
\usepackage[colorinlistoftodos,textsize=small]{todonotes} 
\usepackage{graphicx, color, epsfig, verbatim, subfigure, epsf} 
\usepackage{forest}
\usepackage{float}

\numberwithin{equation}{section}
\newtheorem{theorem}{Theorem}[section]

\newtheorem{lemma}[theorem]{Lemma}

\newtheorem{proposition}[theorem]{Proposition}

\theoremstyle{remark}
\newtheorem{ex}[theorem]{Example}
\newtheorem*{theorem*}{Theorem}

\newtheorem{definition}{Definition}
 
\newcommand{\T}{\mathbb{T}}

	\newcommand{\cc}{\overline}

	\newcommandx{\concern}[2][1=]{\todo[color = red!70!,#1]{Concern: #2}} 
	\newcommandx{\refq}[2][1=]{\todo[color = yellow!40!,#1,]{Reference: #2}} 
	\newcommandx{\wording}[2][1=]{\todo[color = violet!50!,#1,]{Wording: #2}} 
	\newcommandx{\alan}[2][1=]{\todo[color = green!25!,#1]{#2}}
	\newcommandx{\mere}[2][1=]{\todo[color = blue!25!,#1]{#2}}

        \newcommand{\McC}{\raise.5ex\hbox{c}}

\title[Rational inner skew-products]{Dynamics of low-degree rational inner skew-products on $\mathbb{T}^2$}

\author[Sola]{Alan Sola}
\address{Department of Mathematics, Stockholm University, 106 91 Stockholm, Sweden}
\email{sola@math.su.se}

\author[Tully-Doyle]{Ryan Tully-Doyle$^\dagger$}
\address{Department of Mathematics, California Polytechnic State University, San Luis Obispo, CA 93407, USA}
\email{rtullydo@calpoly.edu}
\date{\today}

\thanks{$\dagger$ Partially supported by National Science Foundation DMS Analysis Grant 2055098}
\dedicatory{Dedicated to H\aa kan Hedenmalm on the occasion of his sixtieth birthday}

\begin{document}

\begin{abstract} 
	We examine iteration of certain skew-products on the bidisk whose components are rational inner functions, with emphasis on simple maps of the form $\Phi(z_1,z_2) = (\phi(z_1,z_2), z_2)$. If $\phi$ has degree $1$ in the first variable, the dynamics on each horizontal fiber can be described in terms of M\"obius transformations but the global dynamics on the $2$-torus exhibit some complexity, encoded in terms of certain $\mathbb{T}^2$-symmetric polynomials. We describe the dynamical behavior of such mappings $\Phi$ and give criteria for different configurations of fixed point curves and rotation belts in terms of zeros of a related one-variable polynomial.
 \end{abstract}
\maketitle

\section{Introduction and overview}\label{sec:Intro}
 Iteration of a rational function $R(z)=\frac{q(z)}{p(z)}$ on the Riemann sphere, that is, the study of
\[z\mapsto R^n(z)=(R\circ R \circ \cdots \circ R)(z) \quad (n=1,2,\ldots)\] 
on $\mathbb{C}_{\infty}=\mathbb{C}\cup\{\infty\}$,
is a well-known topic in mathematics, discussed in many textbooks (e.g. \cite{BearBook, CGBook, MilBook}) and illustrated in beautiful computer images. The theory is quite mature, but important new results are still being discovered. The higher-dimensional theory, addressing iteration of $n$-variable polynomial or rational mappings $R$ is of later date but is rapidly developing. See, for instance, \cite{ForSurv,SibSurv, MilBook}, and the references therein, for basic overviews of dynamics in several complex variables.

In a different direction, considering self-maps of special bounded domains in $\mathbb{C}^n$ (the unit disk in the complex plane, the unit $n$-ball) allows for the study of iteration of functions that are not necessarily defined throughout $\mathbb{C}^n$, and leads to interesting boundary phenomena not observed in the unbounded setting. An important example in this latter direction is the classical Denjoy-Wolff theorem \cite{CGBook} concerning fixed points of analytic self-maps of the unit disk; there are many subsequent extensions of the original result to different settings.

Our investigation concerns the study of the dynamics certain self-maps of the unit bidisk
\[\mathbb{D}^2=\{(z_1,z_2)\in \mathbb{C}^2\colon |z_j|<1, j=1,2\}\]
and seeks to establish some basic facts about their iteration theory. There are many works addressing iteration of analytic self-maps of the bidisk and higher-dimensional polydisks (see, for instance, the papers \cite{Aba98, AG16,C32, F05, H54} and the thesis \cite{NThesis} for some background). Our focus is on obtaining detailed results for a restricted class of rational mappings by using elementary means.

We say that a mapping of the form
\[\Phi\colon \mathbb{D}^2\to \mathbb{D}^2\]
\[(z_1,z_2)\mapsto (\phi_1(z_1,z_2), \phi_2(z_1,z_2))\]
is a {\it rational inner mapping} (RIM) if each component $\phi_j$ is a {\it rational inner function} on $\mathbb{D}^2$. A rational inner function in turn is an analytic function of the form
\[\psi(z_1,z_2)=\frac{q(z_1,z_2)}{p(z_1,z_2)},\]
with $q,p\in \mathbb{C}[z_1,z_2]$ and $p(z)\neq 0$ in $\mathbb{D}^2$, which is bounded in $\mathbb{D}^2$ and has unimodular non-tangential boundary values at almost every point $\zeta \in \mathbb{T}^2=\{(\zeta_1, \zeta_2)\in \mathbb{C}^2\colon |\zeta_j|=1, j=1,2\}$. We recall the basic fact that $\mathbb{T}^2$ is the {\it distinguished boundary} of the bidisk, the subset of the boundary $\partial \mathbb{D}^2$ where most interesting function-theoretic phenomena on $\mathbb{D}^2$ are observed and the maximum modulus principle is supported. By a theorem of Knese \cite{Kne15}, if $\phi_j$ is rational inner then the boundary values $\phi_j^*(\zeta)$ exist as unimodular numbers at {\it every} point of $\mathbb{T}^2$, and so we can view a rational inner mapping $\Phi$ as inducing a map $\Phi\colon \overline{\mathbb{D}^2}\to \overline{\mathbb{D}^2}$, with $\Phi$ sending $\mathbb{T}^2$ to $\mathbb{T}^2$. In the same way, the $n$th iterate of $\Phi$
\[\Phi^n(z_1,z_2)=(\Phi\circ \Phi\circ\cdots  \circ\Phi)(z_1,z_2)\]
can be viewed as a mapping of the closure of the bidisk into itself that fixes the distinguished boundary.

Iteration of certain classes of rational inner mappings on the bidisk and on $\mathbb{T}^2$ has been considered in a number of papers. For instance, monomial maps of the form $(z_1,z_2)\mapsto (z_1^{m_1}z_2^{n_1}, z_1^{m_2}z_2^{n_2})$ appear in \cite{F03} and in other works. In \cite{PS08,PR10} (see also \cite{SBJ17} for some applications), the component maps $\phi_j$ are assumed to be of the special type
\[\psi_j(z_1,z_2)=B_{j,1}(z_1)\cdot B_{j,2}(z_2)\]
where $B_{j,k}$ are one-variable {\it finite Blaschke products}. Recall that these are functions $B\colon \mathbb{D}\to\mathbb{D}$ of the form
\[B(z)=e^{i\alpha}z^m\prod_{k=1}^n\frac{z-a_k}{1-\overline{a}_kz},\]
where $\{a_1,\ldots, a_n\}\subset \mathbb{D}$, $\alpha\in \mathbb{R}$, and $m,n\in \mathbb{N}$. It is apparent that a finite Blaschke product extends continuously to the unit circle $\mathbb{T}$, but the dynamics of one-variable Blaschke products nevertheless exhibit complicated features; see for instance \cite[Chapter 15]{MilBook}. Similarly, the works \cite{F03,PS08,PR10} uncover a rich dynamical structure associated with monomial maps and two-dimensional Blaschke products.

In contrast with one-variable Blaschke products, a general rational inner function $\psi=q/p$ in $\mathbb{D}^2$ can have boundary singularities: these occur at points $\tau=(\tau_1,\tau_2)\in \mathbb{T}^2$ where $q(\tau_1,\tau_2)=p(\tau_1,\tau_2)=0$, and represent a genuinely new higher-dimensional phenomenon. The function $\psi$ is in general discontinuous at $\tau \in \mathbb{T}^2$, even though $\psi^{*}(\tau)$ always exists, meaning that a RIM $\Phi$ need not be a continuous self-map of $\overline{\mathbb{D}^2}$.  This fact will be the source of several interesting phenomena that we observe in this paper.

Apart from some preliminary observations concerning RIMs, we mostly restrict our attention to {\it rational inner skew-products} (RISPs). These are rational inner mappings of the special form
\[\Phi(z_1,z_2)=(\phi_1(z_1,z_2), \phi_2(z_2)),\]
with $\phi_1$ rational inner in $\mathbb{D}^2$, and $\phi_2$ a rational inner function in one variable, viewed as a function on $\mathbb{D}^2$ in the obvious way. Skew-products have been studied extensively in the polynomial setting, see e.g. \cite{J99, PR19} and the references therein. In fact, we focus on the very simplest case of skew-mappings
\begin{equation}
(z_1,z_2)\mapsto (\phi(z_1,z_2), z_2),
\label{eq:RISPdef}
\end{equation}
where $\phi$ is a rational inner function in $\mathbb{D}^2$ having {\it bidegree} $(1,n)$, that is, degree $1$ in $z_1$ and degree $n\in \mathbb{N}$ in $z_2$.  (We drop the subscript in the first component to lighten notation.) Skew-products of the form \eqref{eq:RISPdef} fix horizontal lines so that the main dynamics take place along one-dimensional sets, but even this apparently very simple case gives rise to interesting behavior, as is suggested in the images below. The advantage in working with such low-degree RISPs is that the extremely simple classification of fiber dynamics (discussed below) allows us to focus in detail on the global features of iteration on $\mathbb{T}^2$. 

We proceed to describe the contents of our paper. We begin by examining three elementary examples of RISPs in Section \ref{sec:Ex}. These examples serve as guides for our general investigations and illustrate most of the features we uncover in a more general setting. In Section \ref{sec:Prelim} we introduce basic definitions for RIMs and RISPs, record some facts about RIFs and their numerator and denominator polynomials, and discuss iteration of M\"obius transformations in the unit disk. Section \ref{sec:main} contains the main results of our paper. After some preliminary remarks about general RIMs and certain associated essentially $\mathbb{T}^2$-symmetric polynomials, we focus on degree $(1,n)$ RISPs. We then introduce and study rotation belts associated with a RISP, strips on which the dynamical actions on fibers are conjugate to rotations, and give an estimate on their number. We investigate how different components of the fixed point set of a RISPs can come together at a singular fixed point, and give criteria for different configurations to be present. The main tool we use is a one-variable polynomial $Q_{\alpha}$ built from the numerator and denominator polynomials of the first-component map $\phi$. We conclude in \ref{sec:morex} by exhibiting further examples of RISPs having more intricate dynamical behavior, and serving as illustrations of our main results.

\section{Three examples}\label{sec:Ex}

 A fundamental result due to Rudin and Stout (see \cite[Chapter 5]{Rud69}) asserts that any RIF on $\mathbb{D}^2$ can be written as
\begin{equation}\label{eq:rudin}
\phi(z)=e^{i\alpha}z_1^{\beta_1}z_ 2^{\beta_2}\frac{\tilde{p}(z)}{p(z)}
\end{equation}
where $\alpha\in \mathbb{R}$, $\beta_1,\beta_2 \in \mathbb{N}$, the $p\in\mathbb{C}[z_1,z_2]$ has no zeros in $\mathbb{D}^2$, and $\tilde{p}$ is the 
{\it reflection} of  $p$, defined as
\begin{equation}
\tilde{p}(z_1,z_2)=z_1^mz_2^n\overline{p\left(\frac{1}{\bar{z}_1},\frac{1}{\bar{z}_2}\right)}.
\label{eq:preflection}
\end{equation}
Here the pair $(m,n)$ is the bidegree of $p$, that is $p$ has degree $m$ in $z_1$ and degree $n$ in $z_2$. 

This structural result makes examples of rational inner functions particularly easy to construct - all one requires is a polynomial with no zeros in the bidisk. The following such examples can be fruitfully analyzed using elementary means. 
\begin{ex}\label{ex:fave}
\begin{figure}[h!]
    \subfigure[Vertical lines before mapping.]
      {\includegraphics[width=0.4 \textwidth]{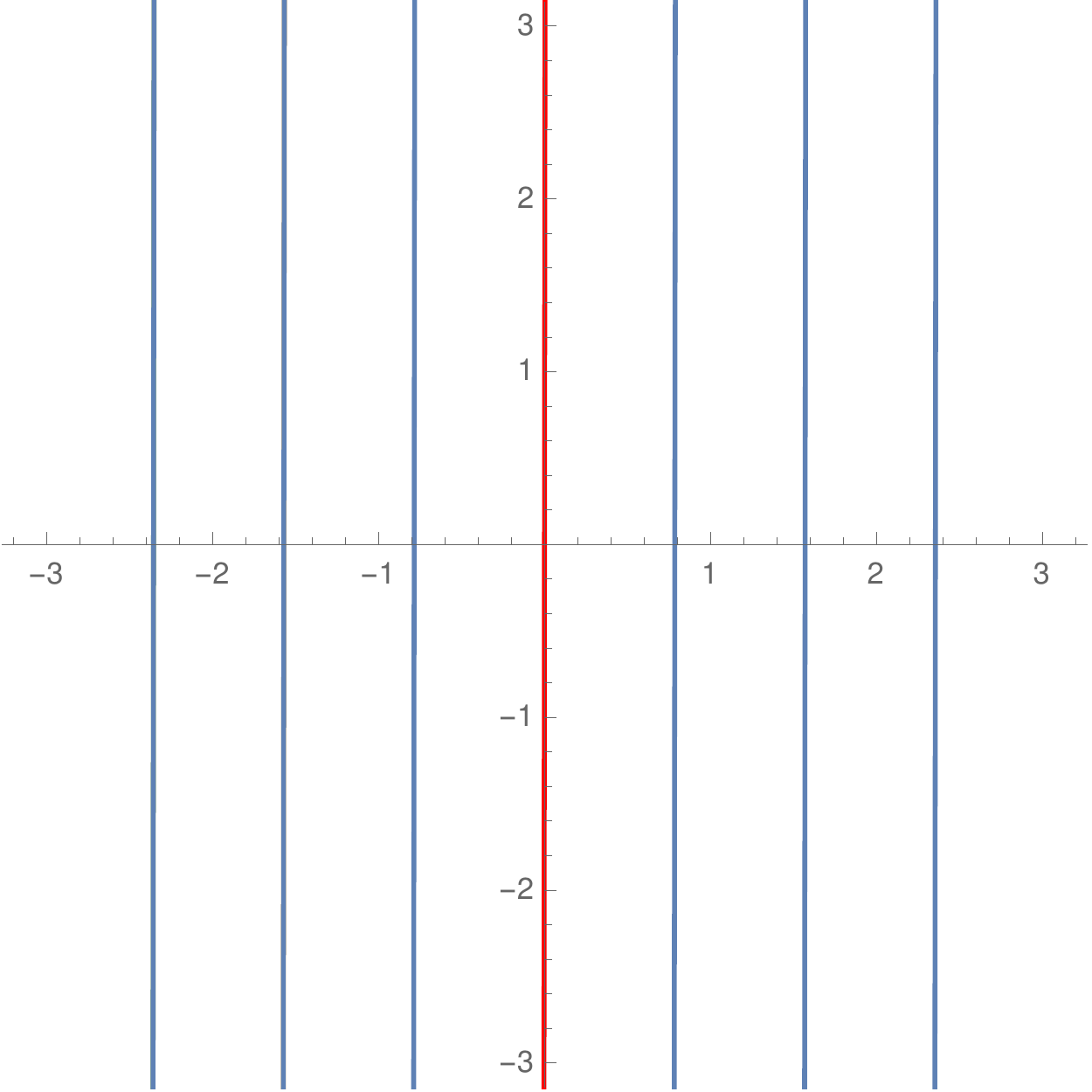}}
    \hfill
    \subfigure[$\Phi^n$ for $n=1$.]
      {\includegraphics[width=0.4 \textwidth]{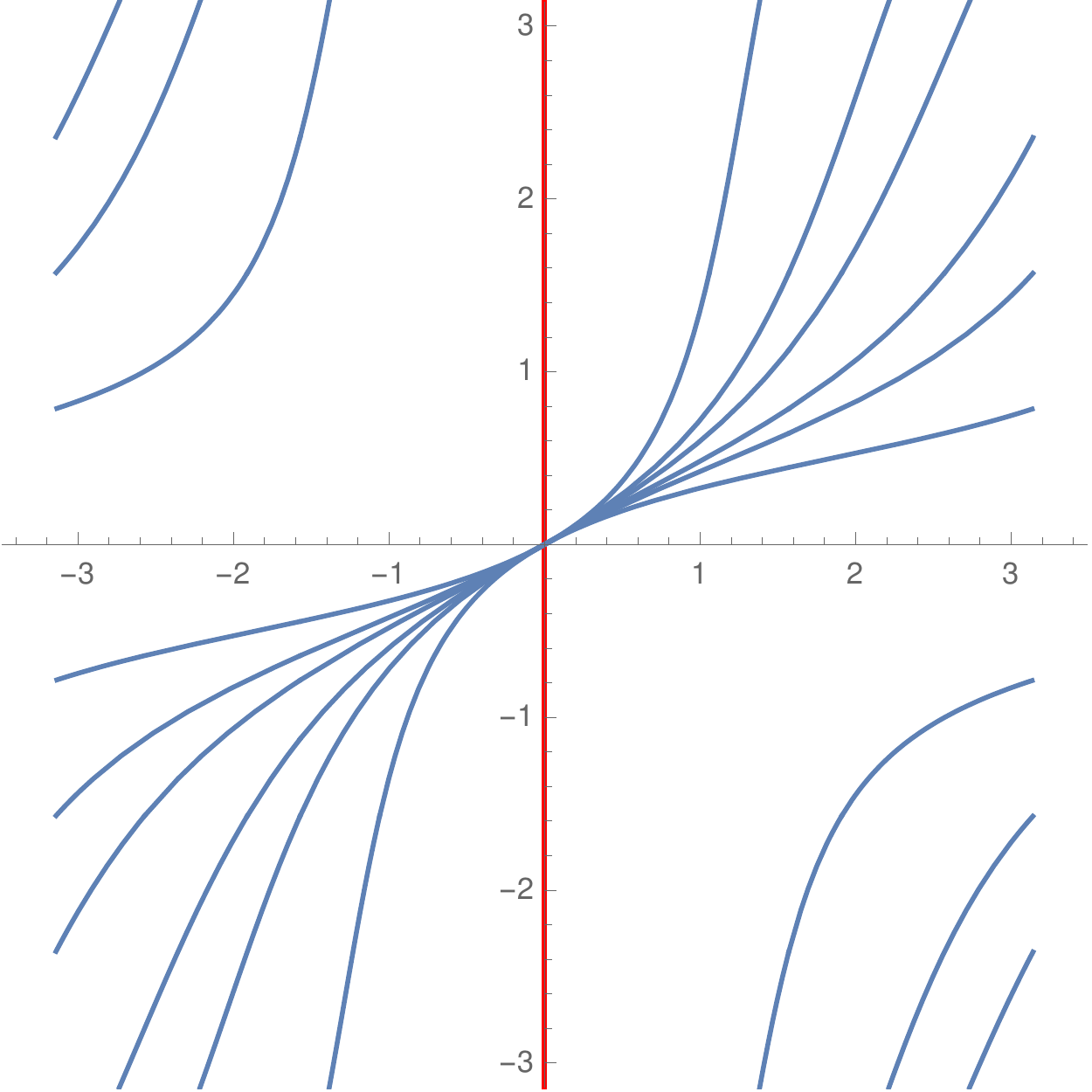}}
      \subfigure[$\Phi^n$ for $n=2$]
      {\includegraphics[width=0.4 \textwidth]{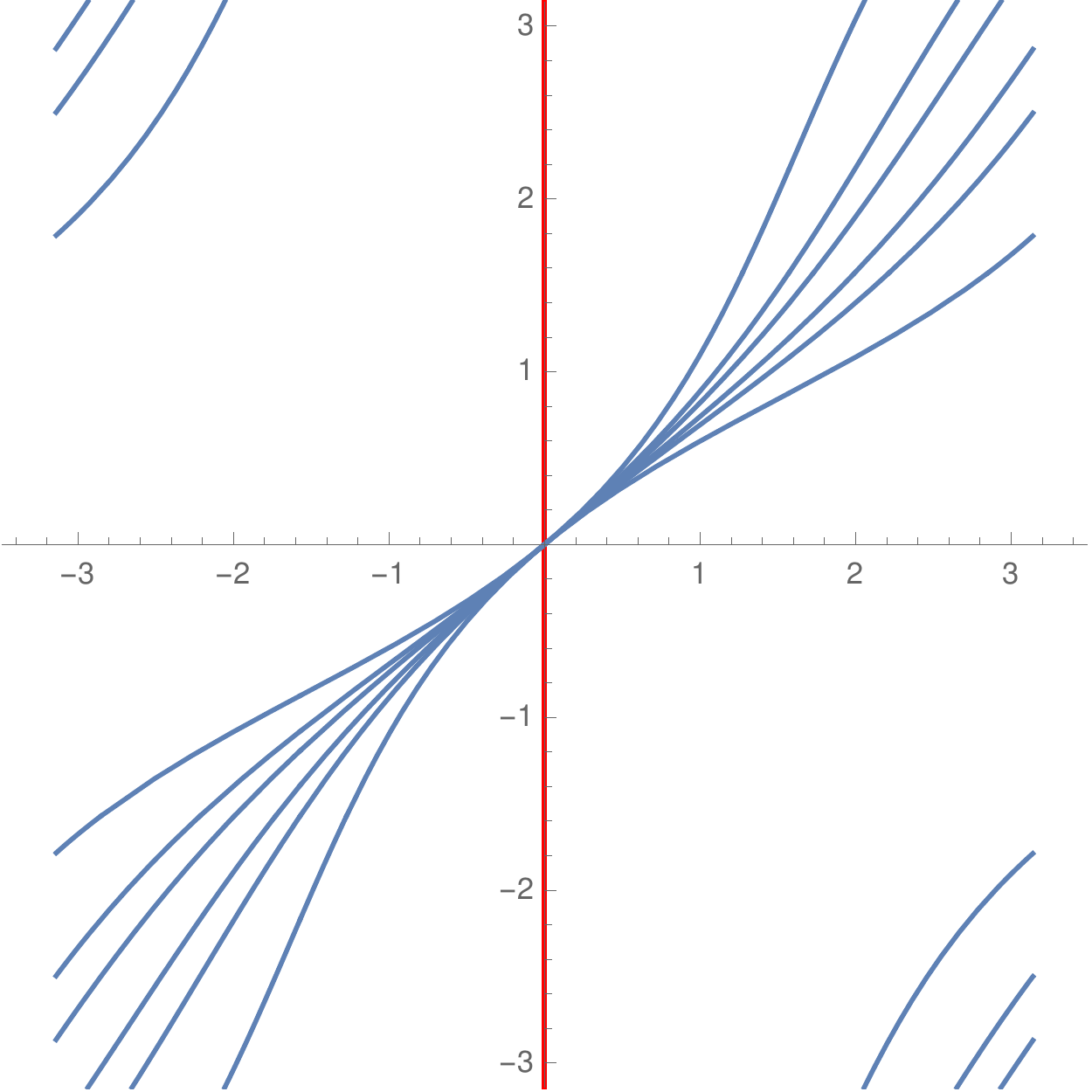}}
    \hfill
    \subfigure[$\Phi^n$ for $n=3$.]
      {\includegraphics[width=0.4 \textwidth]{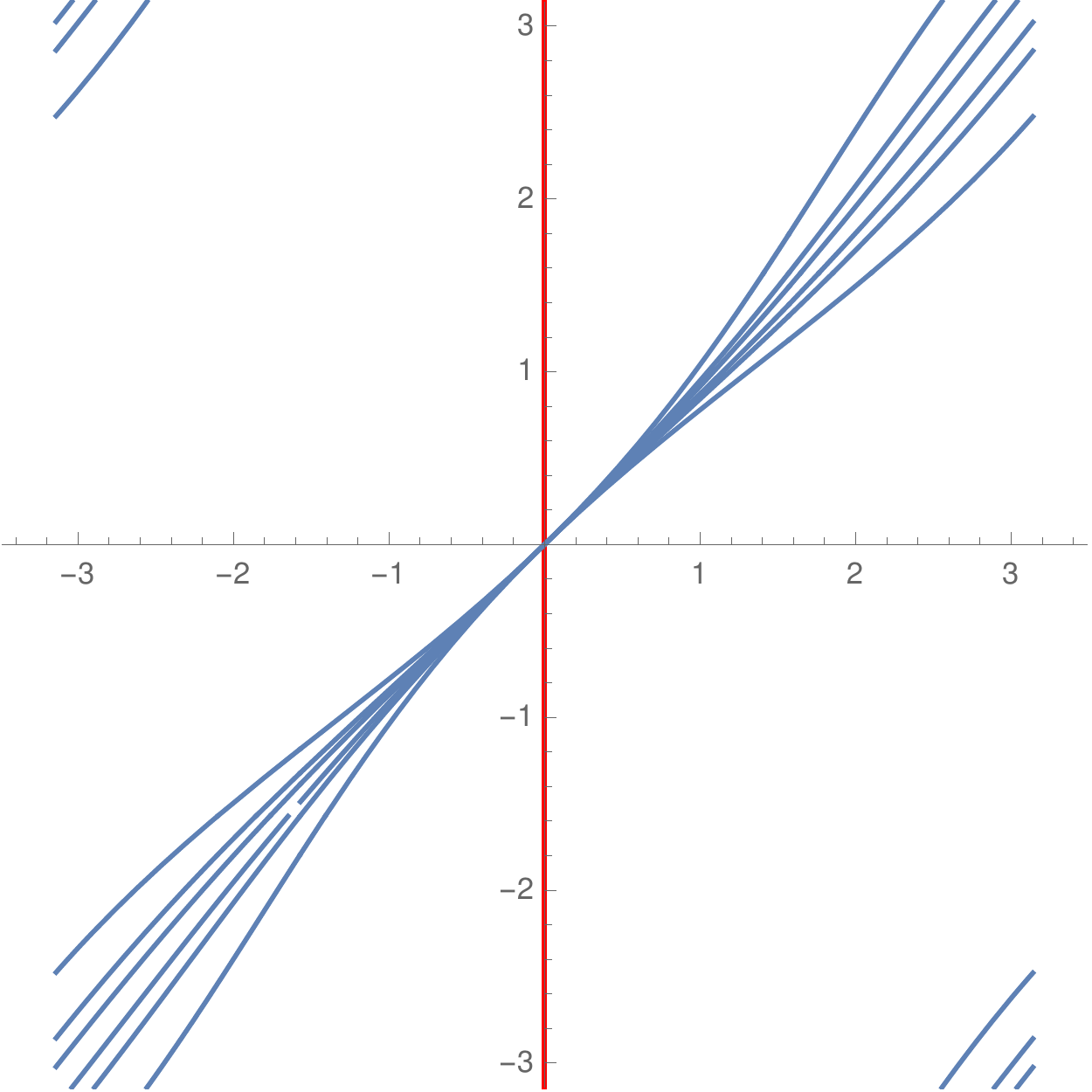}}
      \subfigure[$\Phi^n$ for $n=4$]
      {\includegraphics[width=0.4 \textwidth]{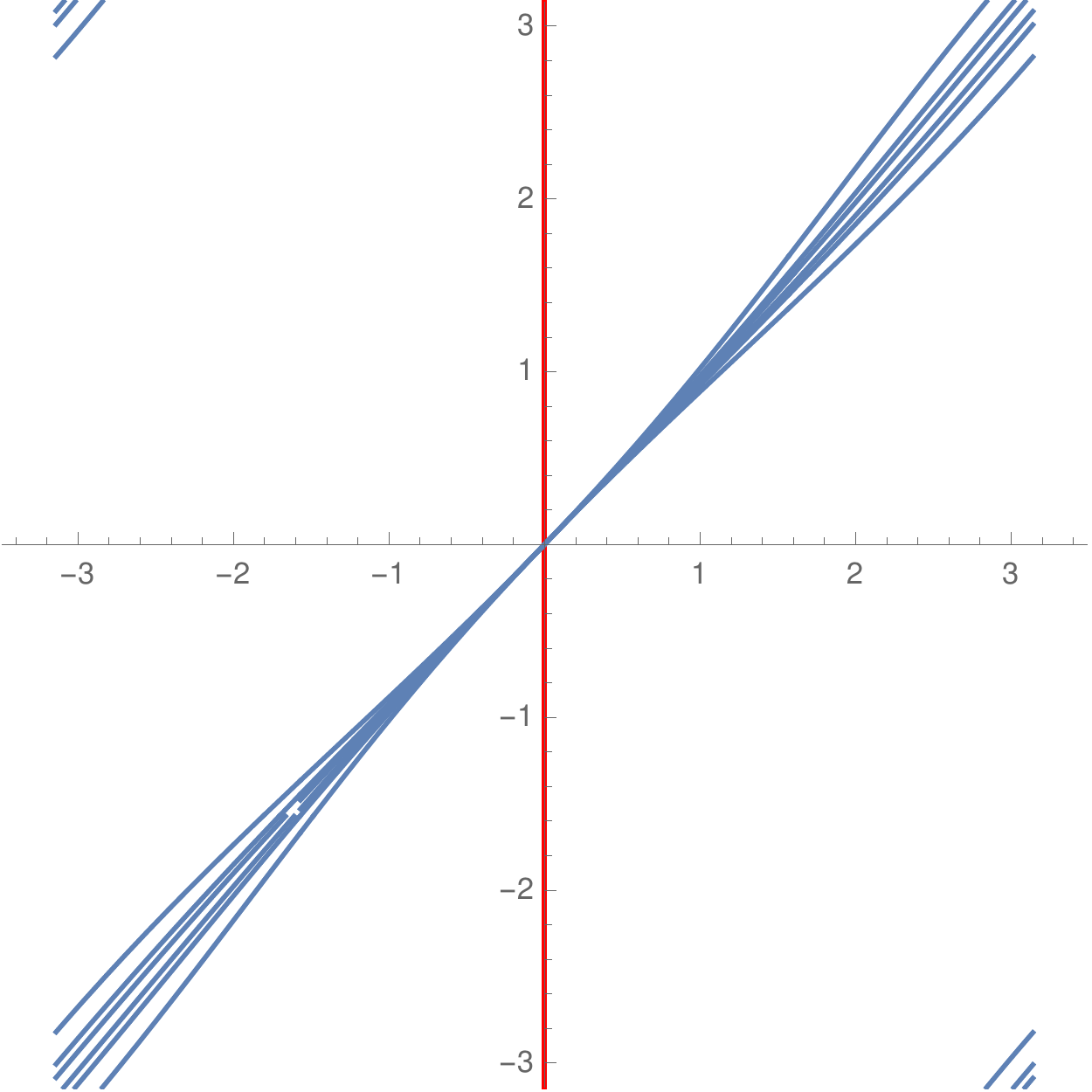}}
    \hfill
    \subfigure[$\Phi^n$ for $n=5$.]
      {\includegraphics[width=0.4 \textwidth]{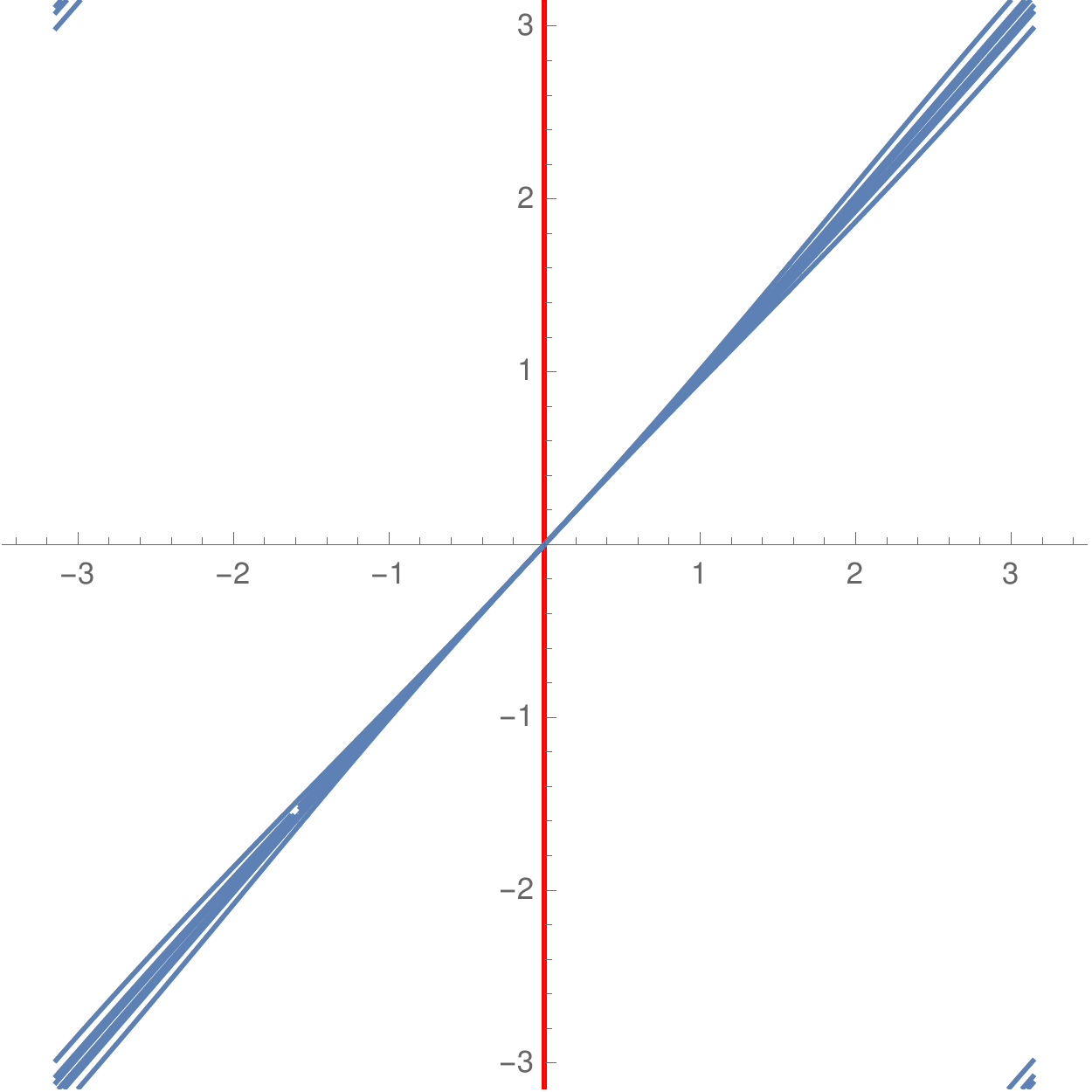}}
  \caption{\textsl{Iteration of $\Phi=(-\frac{2z_1z_2-z_1-z_2}{2-z_1-z_2},z_2)$ on $\mathbb{T}^2$.}}
  \label{faveplots}
  
\end{figure}
Consider the RISP
\[\Phi(z_1,z_2)=\left(-\frac{2z_1z_2-z_1-z_2}{2-z_1-z_2},z_2\right).\]
The first component $\phi=-\frac{\tilde{p}}{p}$ is a degree $(1,1)$ RIF having a single singularity at $(1,1)$. A computation reveals that the boundary value $\phi^{*}(1,1)=1$, and so $(1,1)$ is a fixed point of 
$\Phi$. In fact, we have $\Phi(z_1,1)=(1,1)$, so the entire line $\{z_2=1\}$ is mapped to $(1,1)$. 
 Next, we note that $\phi^*(1,z_2)=1$ and $\phi^*(z_2,z_2)=z_2$ meaning that both the line $\{z_1=1\}$ and the diagonal are comprised of fixed points of $\Phi$. Solving 
$\phi^{*}(z)=z_1$, or more precisely, the equation
\[\tilde{p}(z)-z_1p(z)=0,\]
confirms that these are all the fixed points of $\Phi$. 

We now iterate $\Phi$. We visualize the action of the iterates by viewing the $2$-torus as $[-\pi, \pi]^2$ and applying $\Phi^n$ to vertical lines of the form $\{(a\pi, t_2)\}$ for several choices of $-1<a<1$. Figure \ref{faveplots} suggests that the iterates converge as $n$ grows. 
Using induction, one can show that, for $n=1,2,\ldots$,
\[\Phi^n(z_1,z_2)=\left(-\frac{2^nz_1z_2-z_1-(2^{n}-1)z_2}{2^n-(2^{n}-1)z_1-z_2}, z_2\right).\]
Hence, for all $(z_1,z_2)\in \overline{\mathbb{D}^2}\setminus \{z_1=1\}$, we have
\[\Phi^n(z_1,z_2)\to (z_2,z_2)  \quad \textrm{as} \quad n\to \infty.\]
This means that each point on the diagonal is attractive on its corresponding horizontal fiber, while points on $\{z_1=1\}$ are repelling fixed points. The special fiber $\{z_2=1\}$ is immediately collapsed into $(1,1)$ by $\Phi$, explaining the pinched appearance of the images.

The special role the lines $\{z_1=1\}$ and $\{z_2=1\}$ play here is related to the fact that they are the level sets of $\phi$ corresponding to the non-tangential value $\lambda=1$, which is attained at the singularity of $\phi$ at $(1,1)$ (see \cite{BPS18}). This will be discussed in detail below.
\end{ex}

\begin{ex}\label{ex:para}
\begin{figure}[h!]
    \subfigure[Vertical lines.]
      {\includegraphics[width=0.35 \textwidth]{vertparacurves.pdf}}
\hfill
    \subfigure[$\Phi^n$ for $n=1$.]
      {\includegraphics[width=0.35 \textwidth]{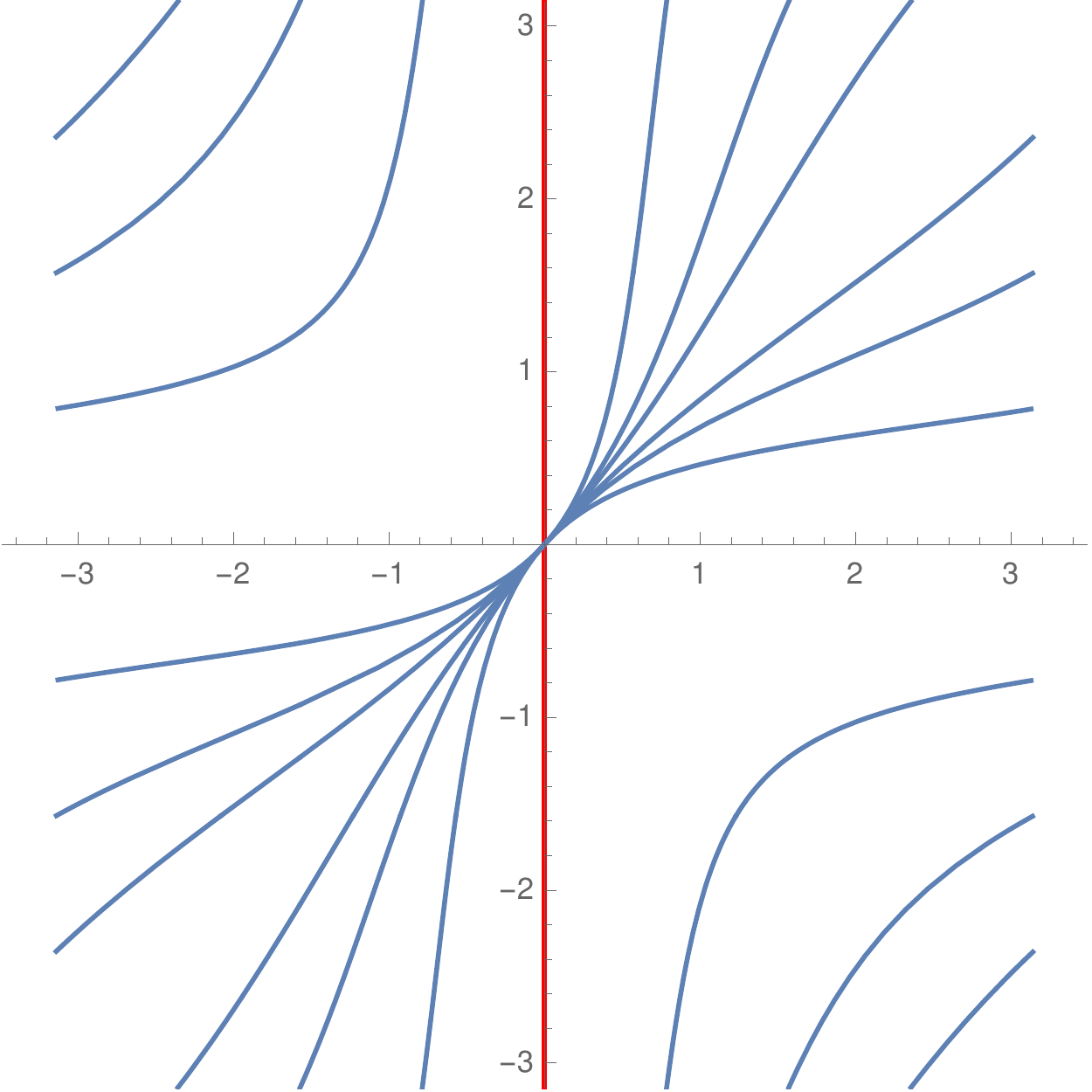}}
      \subfigure[$\Phi^n$ for $n=2$]
      {\includegraphics[width=0.35 \textwidth]{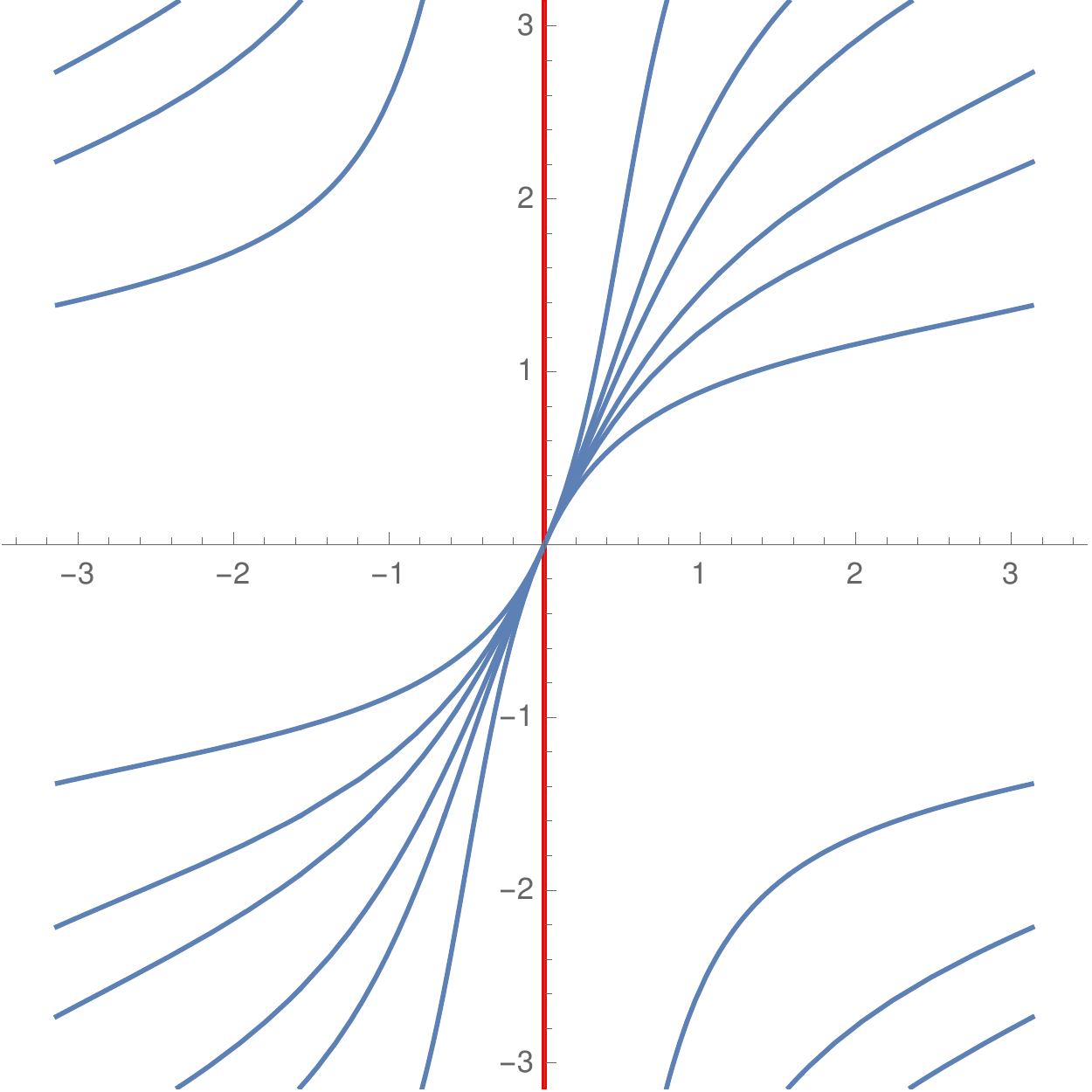}}
    \hfill
    \subfigure[$\Phi^n$ for $n=3$.]
      {\includegraphics[width=0.35 \textwidth]{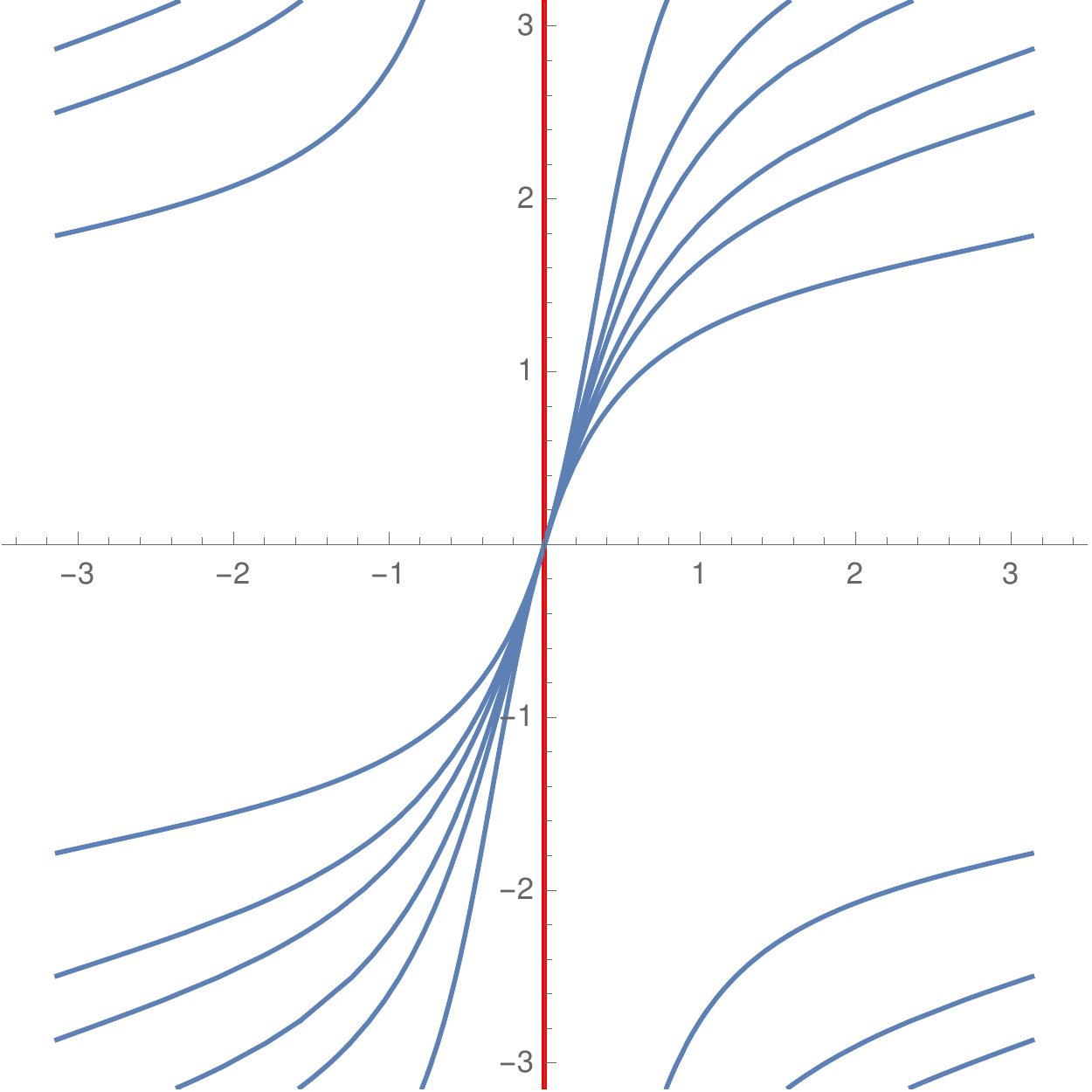}}
      \subfigure[$\Phi^n$ for $n=4$]
      {\includegraphics[width=0.35 \textwidth]{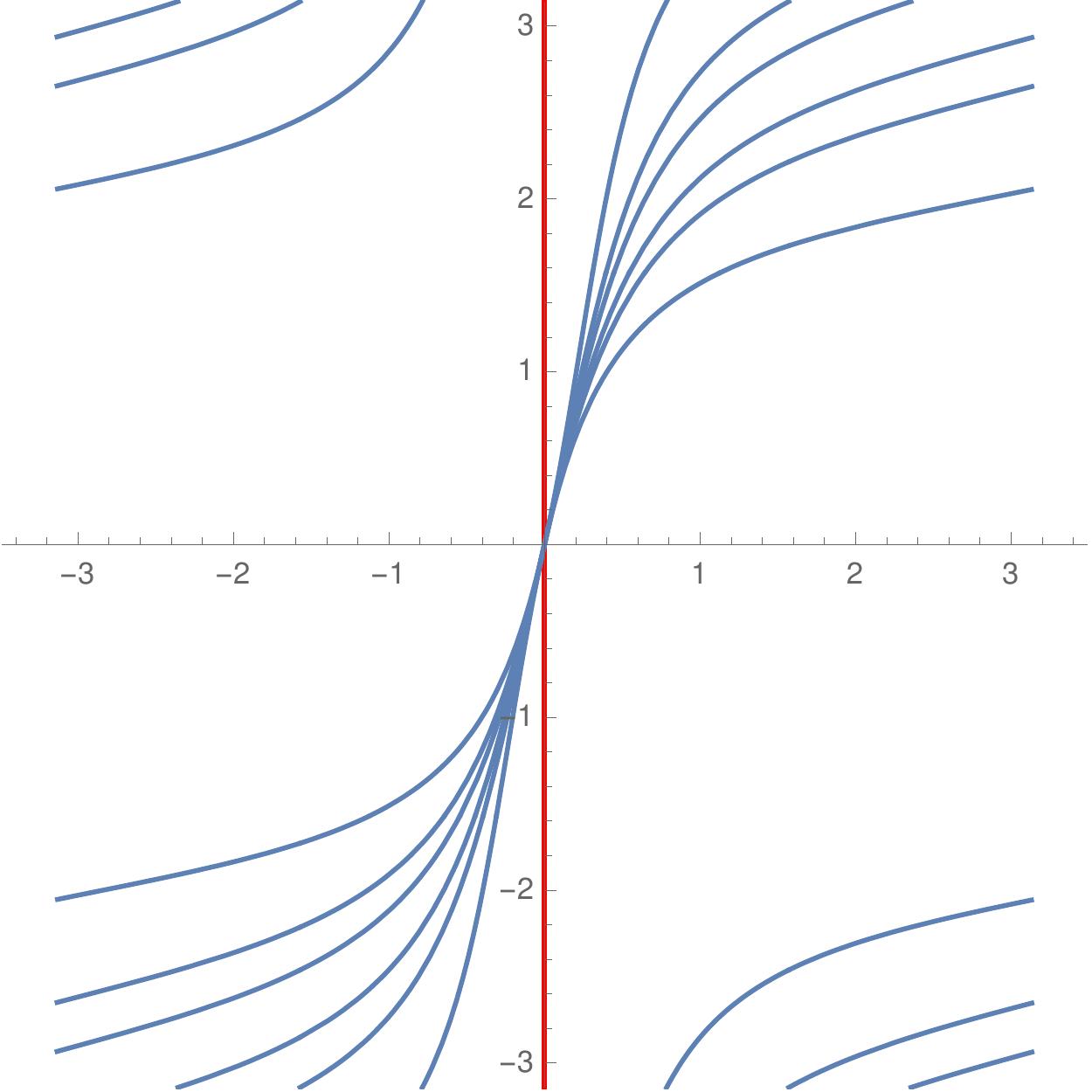}}
    \hfill
    \subfigure[$\Phi^n$ for $n=5$.]
      {\includegraphics[width=0.35 \textwidth]{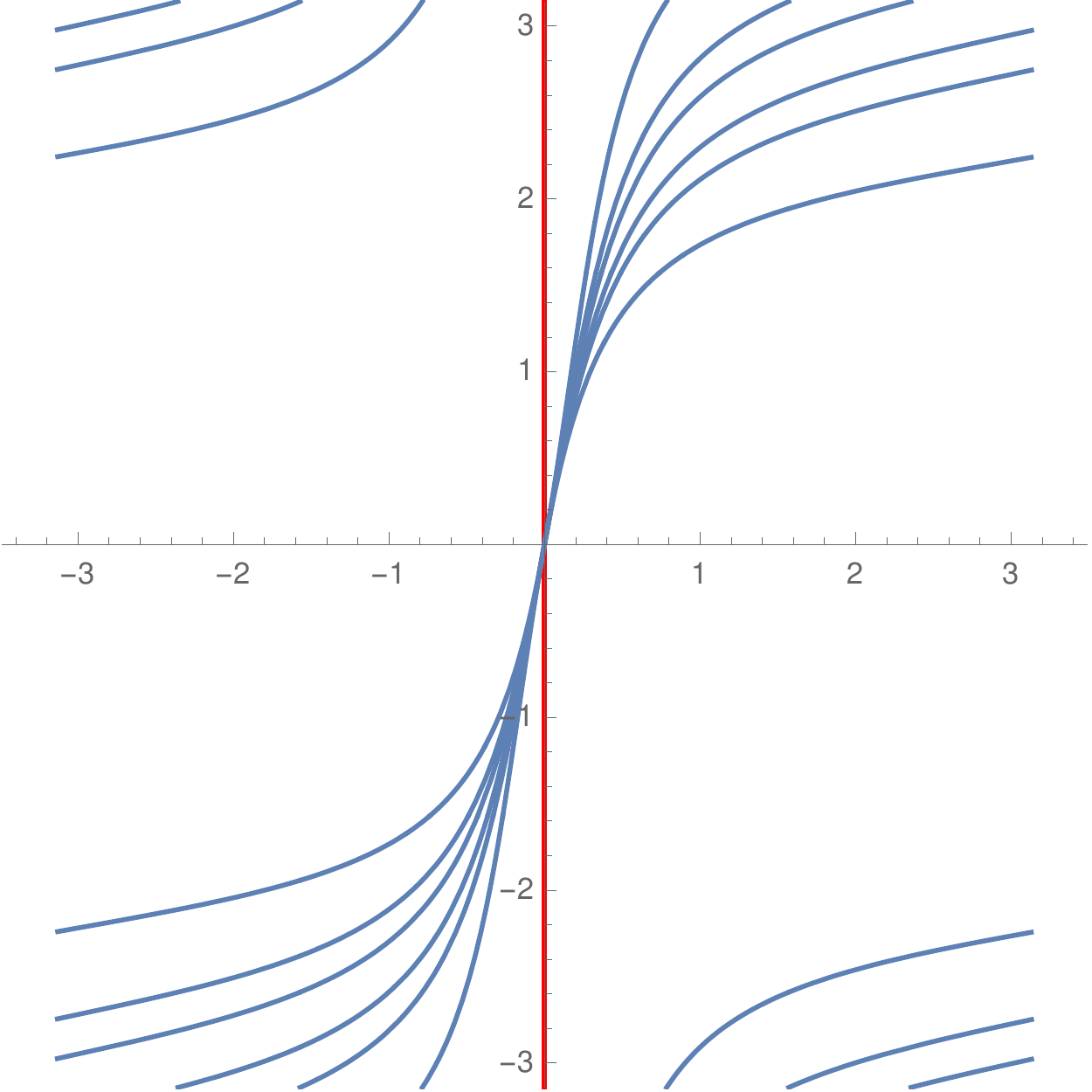}}
  \caption{\textsl{Iteration of $\Phi=(-\frac{3z_1z_2-z_1-z_2-1}{3-z_1-z_2-z_1z_2},z_2)$ on $\mathbb{T}^2$.}}
  \label{parabplots}
\end{figure}
Our next example is
\[\Phi(z_1,z_2)=\left(-\frac{3z_1z_2-z_1-z_2-1}{3-z_1-z_2-z_1z_2},z_2\right).\]
The first component map $\phi=-\frac{\tilde{p}}{p}$ again has a unique singularity on $\mathbb{T}^2$ at the point $(1,1)$, with $\phi^*(1,1)=1$, and since $\Phi(z_1,1)=(1,1)$, the fiber $\{z_2=1\}$ collapses.  The fixed points of $\Phi$ on the $2$-torus are determined by $\tilde{p}-z_1p=0$, and since $\tilde{p}-z_1p=-(z_1-1)^2(z_2+1)$, the fixed point consists of 
$\{z_1=1\}\cup\{z_2=-1\}$, a union of two lines.

Since $\phi(z_1,-1)=z_1$, we obtain $\Phi(z_1,-1)=(z_1,-1)$, confirming that all the points on $\{z_2=-1\}$ are fixed. By contrast, all other fibers contain a single fixed point at $(1,z_2)$. The dynamics of $\Phi$ are shown in Figure 2: we see that when $t_2>0$, points are attracted to the $t_2$-axis from the right, and when $t_2<0$, the fibers are attracted to the axis from the left.  A concrete formula for $\lim_{n\to \infty}\Phi$ is not readily apparent, but a computation shows that
\[\frac{\partial \phi}{\partial z_1}(z_1,z_2)=4\frac{(z_2-1)^2}{(3-z_1-z_2-z_1z_2)^2},\]
hence $\frac{\partial \phi}{\partial z_1}(1,z_2)=1$. As we will see later, this confirms that the fixed points on $\{(1,\zeta_2)\}$ are neither attractive nor repelling in this example: they are so-called parabolic fixed points.
\end{ex}
These two examples give an indication of the features (the presence of a collapsing fiber, fixed point curves being smooth, convergence to fixed points) present the general $(1,n)$ RISP setting. However, other phenomena do arise, and mixed behavior is possible, as the next example shows.

\begin{ex}\label{ex:rot}
The bidegree $(1,1)$ RISP
\[\Phi(z_1,z_2)=\left(\frac{3z_1z_2-z_1-z_2}{3-z_1-z_2},z_2\right)\]
has $\Phi(1,1)=(1,1)$ so that $(1,1)$ is again a fixed point, but in this example,
\[\phi_1(z_1) = \phi(z_1,1)=\frac{2z_1-1}{2-z_1},\]
which is non-constant, meaning that $\{z_2=1\}$ does not collapse into a point under $\Phi$. In fact, the component $\phi$ does not have a singularity at $(1,1)$ or, indeed, at any point of $\mathbb{T}^2$. Instead, the points $(1,1)$ and $(-1,1)$ are fixed by $\Phi$, and since 
$\phi_1'(1)=3$ and $\phi_1'(-1)=\frac{1}{3}$, the fixed point $(1,1)$ is repelling while $(-1,1)$ is attracting. 

We can determine all the fixed points of $\Phi$ on $\mathbb{T}^2$ by solving $\tilde{p}(z)-z_1p(z)=0$ for $z_2$. We obtain
\[z_2=\psi(z_1)=z_1\frac{4-z_1}{4z_1-1},\]
and we note that the right-hand side is unimodular for unimodular $z_1$, so that we again get a curve of fixed points in $\mathbb{T}^2$. However, $\psi$ is not surjective on $\mathbb{T}$, meaning that there are $z_2$-fibers without fixed points; see Figure \ref{rotaplots}(f). For instance, $-1$ is not in the range of $\psi$ on $\mathbb{T}$, and a computation reveals that $\phi(z_1,-1)=-\frac{4z_1-1}{4-z_1}$ satisfies $\phi^2=\mathrm{id}$. Thus, the fiber dynamics associated with $\Phi^n$ on $\{z_2=-1\}$ are those of a rational rotation of order $2$. This is clearly visible when looking at the top of the images in Figure \ref{rotaplots}. Other fibers that do not intersect the fixed point curve experience rotations of different orders.

\begin{figure}[h!]
    \subfigure[$\Phi^n$ for $n=1$.]
      {\includegraphics[width=0.35 \textwidth]{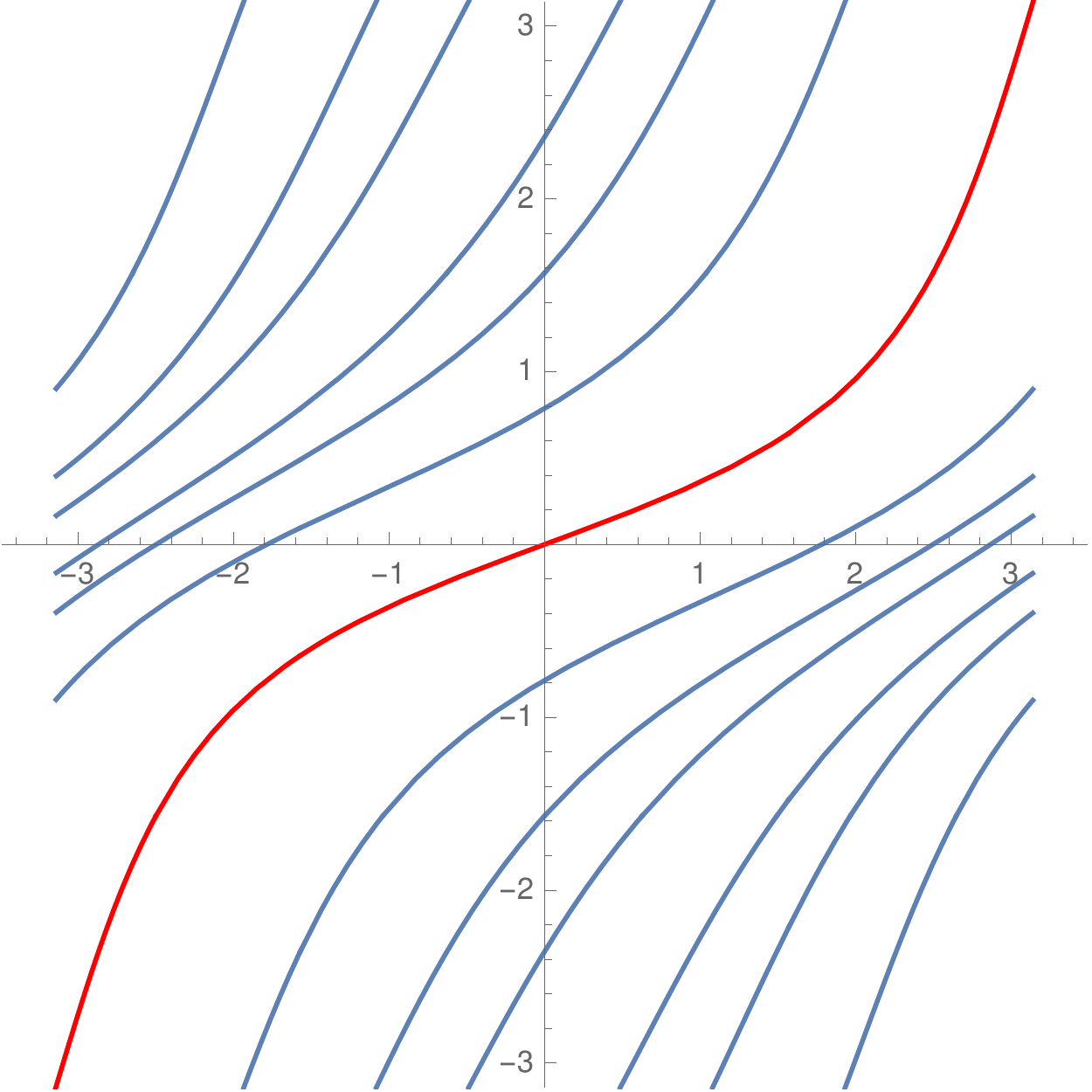}}
      \subfigure[$\Phi^n$ for $n=2$]
      {\includegraphics[width=0.35 \textwidth]{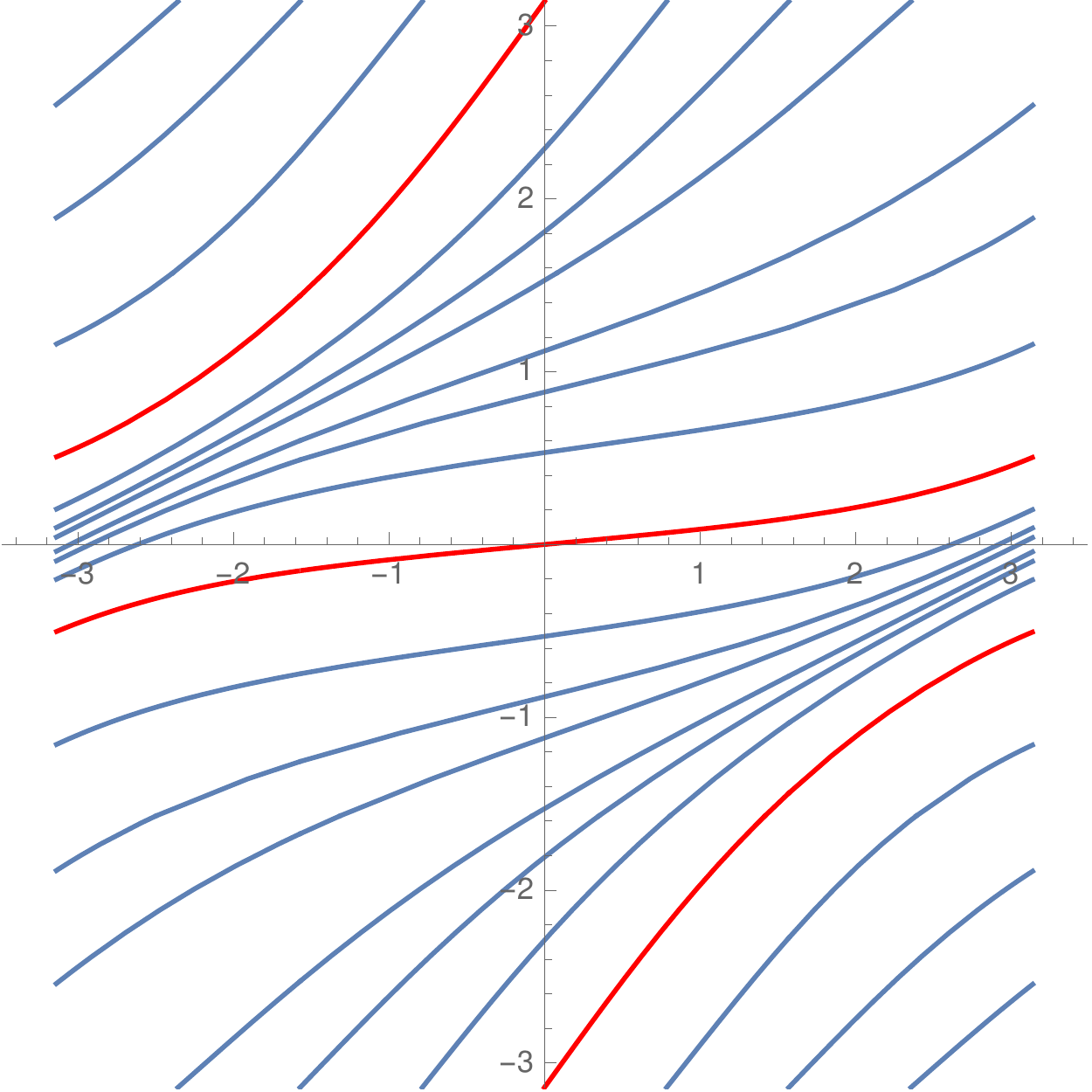}}
    \hfill
    \subfigure[$\Phi^n$ for $n=3$.]
      {\includegraphics[width=0.35 \textwidth]{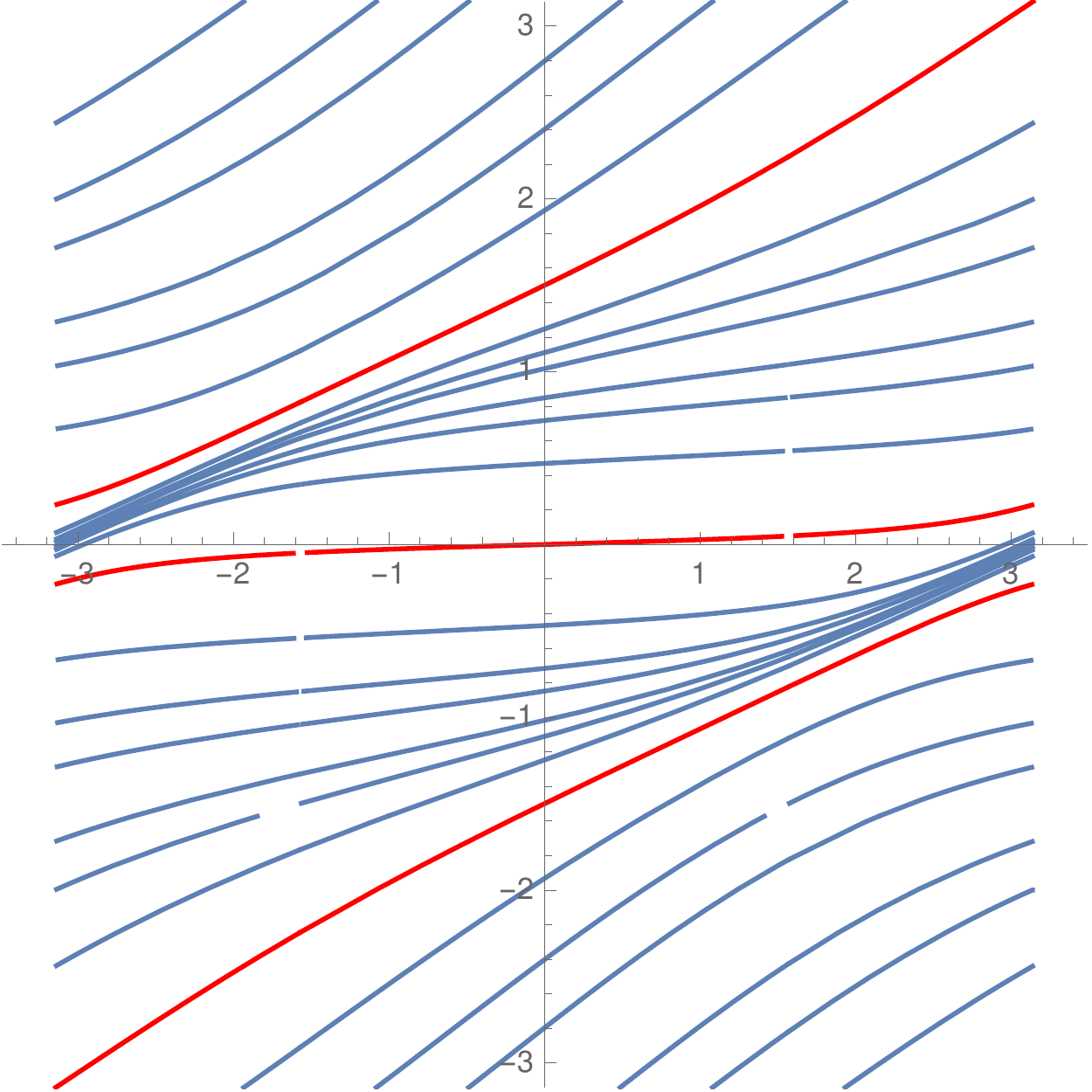}}
      \subfigure[$\Phi^n$ for $n=4$]
      {\includegraphics[width=0.35 \textwidth]{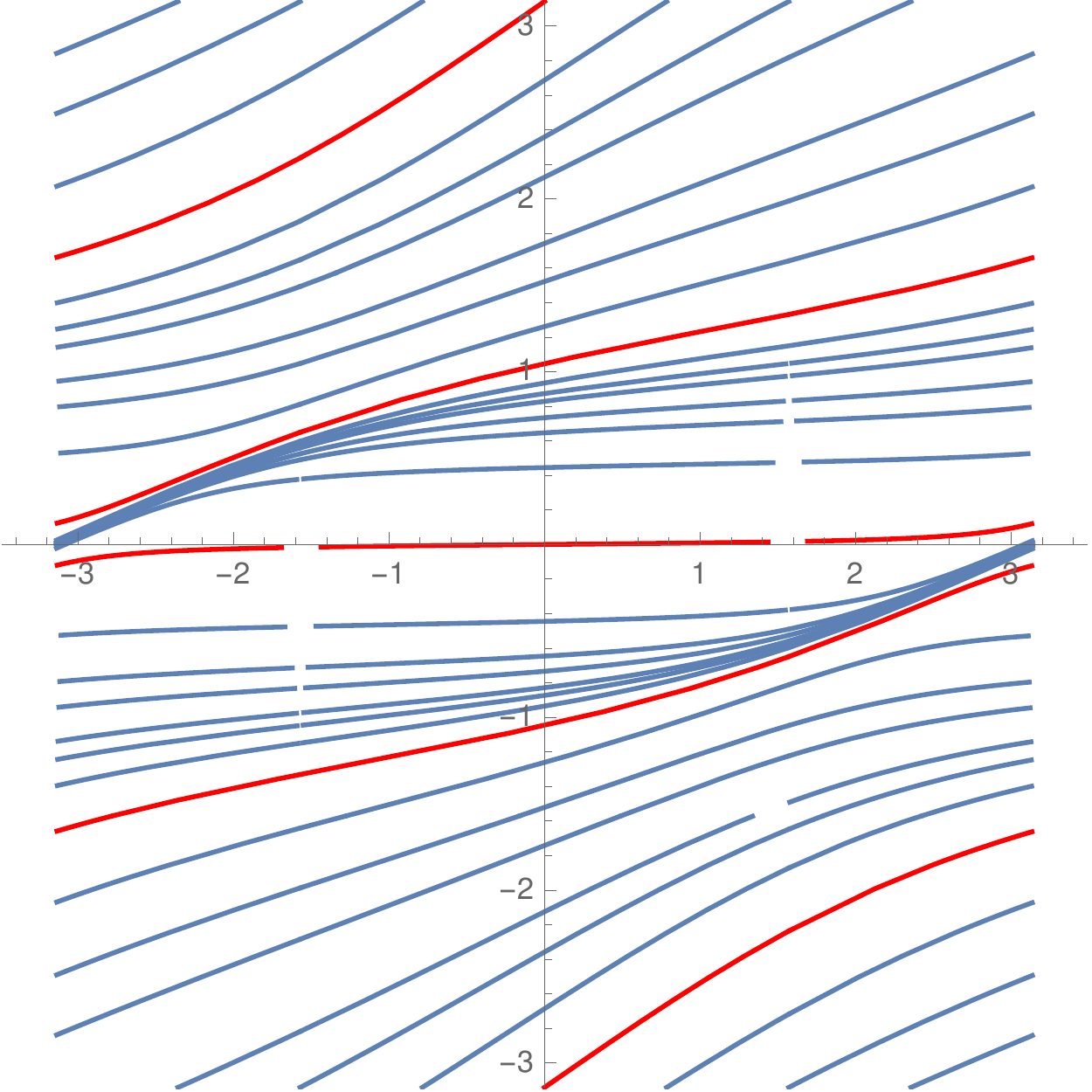}}
    \hfill
    \subfigure[$\Phi^n$ for $n=5$.]
      {\includegraphics[width=0.35 \textwidth]{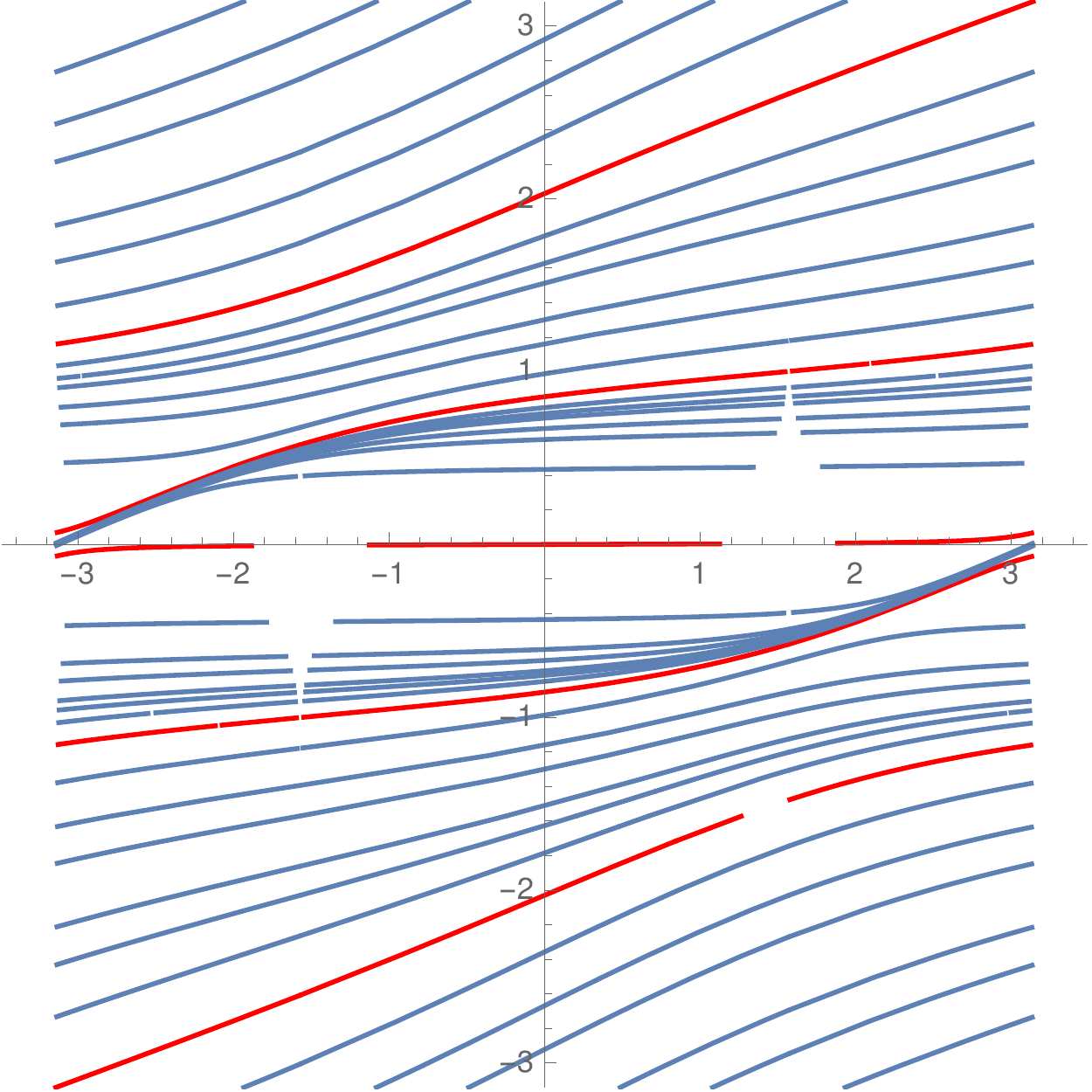}}
    \subfigure[Invariant curve.]
      {\includegraphics[width=0.35 \textwidth]{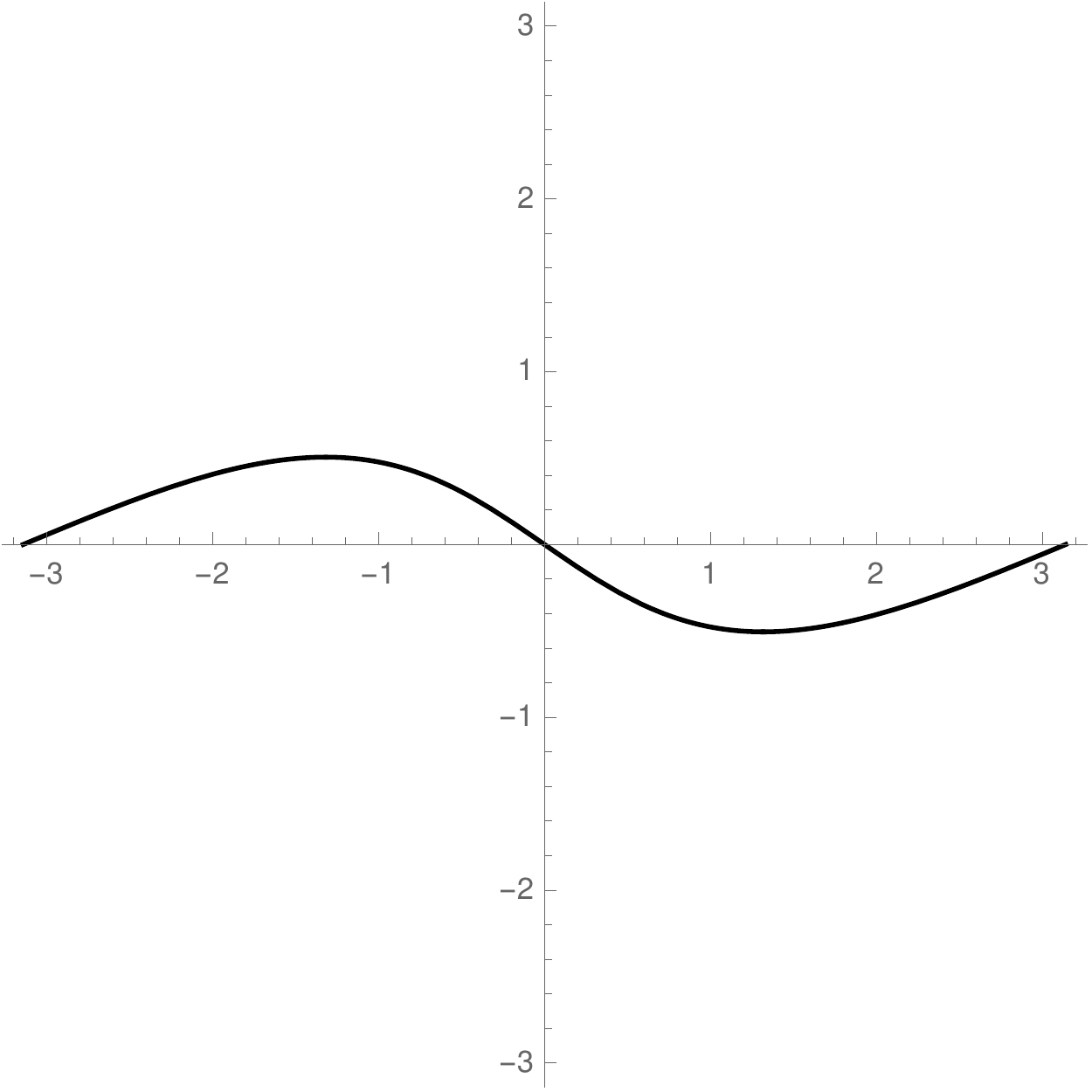}}
  \caption{\textsl{Iteration of $\Phi=(\frac{3z_1z_2-z_1-z_2}{3-z_1-z_2},z_2)$ on $\mathbb{T}^2$.}}
  \label{rotaplots}
\end{figure}
\end{ex}

This third example has some features in common with the first two, but the dynamics varies in nature across fibers. Together, our examples illustrate the three possible dynamical behaviors of a bidegree $(1,n)$ RISP -- attracting/repelling fixed points, so-called parabolic fixed points, and rotations on a $z_2$-fiber. In the investigation below, we explore how the actions on individual fibers fit together in a global picture of the dynamical behavior of a bidegree $(1,n)$ RISP on $\mathbb{T}^2$. 

\section{Preliminaries}\label{sec:Prelim}
\subsection{Rational inner functions on the bidisk}
 
Recall that a RIF is of the form
\[
\phi(z)=e^{i\alpha}z_1^{\beta_1}z_ 2^{\beta_2}\frac{\tilde{p}(z)}{p(z)}.
\] In what follows, we shall frequently assume that our RIFs are normalized. We take this to mean $\beta_1=\beta_2=0$, meaning that $$\phi=e^{i\alpha}\frac{\tilde{p}}{p}$$ for some polynomial $p$ with no zeros in $\mathbb{D}^2$. We also assume that $p$ is {\it atoral} in the sense of \cite{AMS06}, meaning that $p$ and $\tilde{p}$ have no common factors. (If $p$ is not initially atoral, then any toral factors can be canceled with the corresponding toral factors in $\tilde{p}$.)

Let us review some basic facts concerning rational inner functions on the bidisk. By the definition of the reflection operation, any $\tau \in\mathbb{T}^2$ that is a zero of $p$ is a zero of $\tilde{p}$. We say that $\tau\in\mathbb{T}^2$ is a {\it singular point} of $\phi$ if $p(\tau)=0$ 
(and hence also $\tilde{p}(\tau)=0$). Next, let us recall {\it B\'ezout's theorem} for $\mathbb{C}_{\infty}\times \mathbb{C}_{\infty}$, as discussed in \cite[Section 12]{Kne15}. Namely, if 
$P, Q\in \mathbb{C}[z_1,z_2]$ are polynomials with no common factors and with $\operatorname{deg}P=(M_1,N_1)$ and $\operatorname{deg} Q=(M_2,N_2)$, then $P$ and $Q$ have $M_1N_2+M_2N_1$ common zeros in $\mathbb{C}_{\infty}\times \mathbb{C}_{\infty}$. Applying this result to $P=p$ and $Q=\tilde{p}$ shows that any two-variable rational inner function has at most finitely many singularities on $\mathbb{T}^2$.

Since $\phi$ is a bounded analytic function, Fatou's theorem for polydisks guarantees the existence of non-tangential limits of $\phi$: that is,
\[\phi^*(\zeta_1,\zeta_2)=\angle\lim_{\mathbb{D}^2\ni (z_1,z_2)\to (\zeta_1,\zeta_2)}\phi(z_1,z_2)\]
exists at almost every $\zeta \in \mathbb{T}^2$. Here, $\angle\lim_{z\to \zeta}f(z)$ denotes letting $z$ tend to $\zeta \in \mathbb{T}^2$ with $|z_j-\zeta_j|<c(1-|z_j|)$, $j=1,2$, for some constant $c>1$.
 Knese proved  \cite[Corollary 14.6]{Kne15} that if $\phi$ is rational inner function, then the non-tangential limit $\phi^*(\zeta)$ exists and is unimodular at {\it every} point $\zeta \in \T^2$. However, if a normalized $\phi$ has singular points, then  $\phi^*(\zeta)$ is not typically continuous on $\mathbb{T}^2$. See \cite{AMY12,TD17, BPS18, BPS20} for more background material on RIFs.

\begin{definition}
Let $\Phi=(\phi_1, \phi_2)\colon \mathbb{D}^2\to \mathbb{D}^2$ be a non-constant normalized RIM. We say that $\tau\in \mathbb{T}^2$ is a {\it singular fixed point} (SF-point) of 
$\Phi$ if 
\begin{itemize}
\item $\tau \in\mathbb{T}^2$ is a singular point of at least one of $\phi_1$ and $\phi_2$;
\item $\tau$ is a {\it fixed point}: $\Phi^*(\tau)=\tau$.
\end{itemize}
\end{definition}
Note that, for each denominator $p$, we can make sure the corresponding component RIFs $\phi_j$ satisfy the condition $\phi^*_j(\tau)=\tau_j$ by a suitable choice of unimodular factor $e^{i\alpha_j}$. If $\phi_j$ possesses multiple singularities then we cannot typically normalize all of them to be SF-points simultaneously.

We record a few further facts about rational inner functions. Suppose $\phi$ is RIF of bidegree $(m,n)$. For a fixed $\zeta_2\in \mathbb{T}$, the function 
\[\phi_{\zeta_2}\colon z_1\mapsto \phi(z_1,\zeta_2)\] is a bounded rational function in the unit disk, attaining unimodular boundary values at every point of $\mathbb{T}$. Hence $\phi_{\zeta_2}$ is constant, a monomial, or a finite Blaschke product of degree $1\leq d\leq n$. Returning to the setting of a RISP $\Phi=(\phi, z_2)$, we note that if $\lambda$ is a constant, $\Phi$ maps the set
\begin{equation}
F_{\lambda}=\{(z_1,z_2)\in \overline{\mathbb{D}^2}\colon  z_2=\lambda\}
\label{fiber}
\end{equation}
into itself. The set $F_{\lambda}$ is referred to as a {\it fiber}. On each $F_{\lambda}$, the first component $\phi_{\lambda}$ is generically a finite Blaschke product of degree $n$ but on certain fibers the degree may drop. The most extreme examples are the following.
\begin{definition}
Let $\Phi=(\phi,z_2)$ be a normalized RISP, and let $\lambda \in\mathbb{T}$. 
We say that $F_{\lambda}$ is a {\it collapsing fiber} for $\Phi$ if the one-variable function $\phi_{\lambda}$ is constant. 
\end{definition}
 
We often specialize to bidegree $(1,n)$ RIFs. In that case, we write $\phi=e^{i\alpha}\frac{\tilde{p}}{p}$, where
\begin{equation}
p(z)=p_1(z_2)+z_1p_2(z_2), 
\label{def:p1p2polys}
\end{equation}
for some $p_1,p_2\in \mathbb{C}[z_2]$. Then as a consequence $\tilde{p}(z)=\tilde{p}_2(z_2)+z_1\tilde{p}_1(z_2)$, where each $\tilde{p}_j(z_2)=z_2^n\overline{p}(1/\bar{z}_2)$.  Note that $p_1(z_2)\neq 0$  for $z_2 \in \mathbb {D}$ since $p$ is assumed to have no zeros in the bidisk, and in fact $p_1(z_2)\neq 0$ for $z\in \overline{\mathbb{D}}$ by \cite[Lemma 10.1]{Kne15}. 
Finally, for $\alpha\in \mathbb{R}$, let us define the polynomial
\begin{equation}
Q_{\alpha}(z_2)=(p_1(z_2)-e^{i\alpha}\tilde{p}_1(z_2))^2+4e^{i\alpha}p_2(z_2)\tilde{p}_2(z_2)
\label{def:qapoly}
\end{equation}
which we shall later use to analyze the fixed points of the corresponding RISP $(\phi, z_2)$.
Note that if $\phi$ has bidegree $(1,n)$, then $Q_{\alpha}$ has degree at most $2n$, and that $\tilde{Q}_{\alpha}(z_2)=e^{-2i\alpha}Q_{\alpha}(z_2)$. 

When $\phi$ is a degree $(1,n)$ RIF and $\zeta_2\in \mathbb{T}$, the one-variable function $\phi_{\zeta_2}$ has $\deg(\phi_{\zeta_2})=1$, or is a constant and thus has a collapsing fiber $F_{\zeta_2}$. In this case, we can characterize collapsing fibers using a result from \cite{BCSprep}: reversing the roles of $z_1$ and $z_2$, \cite[Theorem 3.3]{BCSprep} yields the following.

\begin{lemma}\label{lem:collapsefiberlem}
Suppose $\Phi=(\phi,z_2)$ is a normalized bidegree $(1,n)$ RISP. Then $F_{\lambda}$ is a collapsing fiber if and only if $(\tau_1, \lambda)\in \mathbb{T}^2$ is a singularity of $\phi$ for some choice of $\tau_1$. If this is the case, then $\phi^*(\tau_1, \lambda)=\tau_1$ and $(\tau_1, \lambda)$ is an SF-point of $\Phi$. 
\end{lemma}
In other words, Lemma \ref{lem:collapsefiberlem} relates collapsing fibers of a RISP to the level sets of its first component map $\phi$. In this regard, two-variable rational inner functions exhibit better behavior than rational inner functions in $\mathbb{D}^n$ when $n\geq 3$. In \cite[Theorem 2.8]{BPS20}, it is shown that for each fixed $\lambda\in \mathbb{T}$, the unimodular level set 
\begin{equation}
\mathcal{C}_{\lambda}=\{(\zeta_1,\zeta_2)\in \mathbb{T}^2\colon \tilde{p}(\zeta_1,\zeta_2)-\lambda p(\zeta_1,\zeta_2)=0\}
\label{eq:levelcurve}
\end{equation}
consists of components that can be parametrized using analytic functions. By contrast, unimodular level sets in higher dimensions need not even be continuous \cite{BPSajm}.

\subsection{Dynamics of M\"obius transformations}\label{subsec:Mob}
Recall that the elements of the conformal automorphism group $\mathrm{Aut}(\mathbb{D})$ of the unit disk are {\it M\"obius transformations} of the form
\[\mathfrak{m}(z)=e^{i\alpha}\frac{a-z}{1-\overline{a}z},\]
where $\alpha \in \mathbb{T}$ and $a\in \mathbb{D}$. Note that each M\"obius transformation extends to the closed bidisk $\overline{\mathbb{D}}$ and in particular furnishes a smooth diffeomorphism of the unit circle $\mathbb{T}$.

 If $\phi$ is a bidegree $(1,n)$ rational inner function, then for each fixed $\lambda \in \mathbb{T}$, the analytic function $z_1\mapsto \phi(z_1,\lambda)$ is bounded, satisfies $|\phi(z_1,\lambda)|=1$, and is of degree $0$ or $1$. Then either $\phi(\cdot, \lambda)$ is a M\"obius transformation for some values of $\alpha(\lambda)$ and $a(\lambda)$, or $\phi(\cdot, \lambda)$ is constant. The latter possibility corresponds to a collapsing fiber, and by Lemma \ref{lem:collapsefiberlem}, this occurs if only if $(\zeta_1,\lambda)$ is a singularity of $\phi$ for some choice of $\zeta_1$.

In summary, for all but finitely many values 
\[\Lambda^{\flat}=\{\lambda^{\flat}_1,\ldots, \lambda^{\flat}_N\} \subset \mathbb{T},\] 
whose number is bounded by the degree of $\phi$ in $z_2$, the fiber map
\[z_1\mapsto \phi_{\lambda}(z_1)=\phi(z_1,\lambda)\] 
is a M\"obius transformation with parameters $\alpha(\lambda)$ and $a(\lambda)$. The functions $\alpha(\lambda)$ and $a(\lambda)$ are continuous in any open set that does not intersect $\Lambda^{\flat}$. Associated with $\Lambda^{\flat}$ is the set
\begin{equation}
\iota(\Lambda^{\flat})=\bigcup_{\lambda \in \Lambda^{\flat}}F_{\lambda};
\label{badfibers}
\end{equation}
note that this is a finite union of horizontal lines in $\mathbb{T}^2$.

Let us recall the well-known classification of M\"obius transformations depending on their fixed points. See \cite[Chapter 1]{BearBook} or \cite[Chapter 1]{MilBook} for a background discussion.
\begin{definition}
Suppose $\mathfrak{m}\in \mathrm{Aut}(\mathbb{D})\setminus\{\mathrm{id}\}$. Then one of the following holds:
\begin{itemize}
\item $\mathfrak{m}$ is {\it elliptic}: $\mathfrak{m}$ has one fixed point in $\mathbb{D}$, and no fixed points on $\mathbb{T}$
\item $\mathfrak{m}$ is {\it parabolic}: $\mathfrak{m}$ has no fixed points in $\mathbb{D}$ and one fixed point on $\mathbb{T}$
\item $\mathfrak{m}$ is {\it hyperbolic}: $\mathfrak{m}$ has no fixed points in $\mathbb{D}$ and two fixed points on $\mathbb{T}$. 
\end{itemize}
\end{definition}
Fixed points are classified according to the following scheme.
\begin{definition}
Suppose $t \in \overline{\mathbb{D}}$ is a fixed point of a M\"obius transformation $\mathfrak{m}$ and define the {\it multiplier} $\mathfrak{p}=\mathfrak{m}'(t)$. We say that 
\begin{itemize}
\item $t$ is {\it attracting} if $0< |\mathfrak{p}|<1$
\item $t$ is {\it repelling} if $|\mathfrak{p}|>1$
\item $t$ is {\it indifferent} if $|\mathfrak{p}|=1$. 
\end{itemize}
We further distinguish between {\it rationally indifferent fixed points} arising when $\mathfrak{p}$ is a root of unity, and {\it irrationally indifferent fixed points} corresponding to $\mathfrak{p}$ unimodular but not a root of unity.
\end{definition}
It is easy to give examples of M\"obius transformations with these properties. The rotation $\mathfrak{m}(z)=e^{i\alpha}z$ has a rationally or irrationally indifferent fixed point at $t=0$ according to whether $\alpha$ is rational or irrational. The map $\mathfrak{m}(z)=\frac{(2+i)z+i}{2-i-iz}$ has a single fixed point at $t=-1$, and is parabolic.
Finally, the hyperbolic map $\mathfrak{m}(z)=\frac{7z_1-1}{7-z_1}$ has one attracting fixed point at $t_1=-1$ and one attracting fixed point at $t_2=1$.

The dynamics of a M\"obius transformation can be described depending on its type in the above list. Recall that two analytic functions $f,g\colon \mathbb{D}\to \mathbb{D}$ are {\it conformally conjugate} if there exists a conformal map $\mathfrak{n}\colon \mathbb{D}\to \mathbb{D}$ such that \[g=\mathfrak{n}\circ f\circ \mathfrak{n}^{-1}.\]
See \cite[Chapter 2]{BearBook} or \cite[Chapters 8-9]{MilBook} for general background.
Note that conjugation preserves the multiplier of a M\"obius transformation at a fixed point, leading to the following classification.
\begin{proposition}\label{prop:conjfixdescr}
Let $\mathfrak{m} \in \mathrm{Aut}(\mathbb{D})\setminus\{\mathrm{id}\}$.  
Then the following holds:
\begin{itemize}
\item If $\mathfrak{m}$ is elliptic, then $\mathfrak{m}$ is conjugate to a rotation $z\mapsto e^{ia}z$ with angle $a=-i\log \mathfrak{p}$, where $\mathfrak{p}$ is the multiplier at the unique fixed point $t\in \mathbb{D}$.
\item If $\mathfrak{m}$ is parabolic then $\mathfrak{m}$ has multiplier equal to $1$ at its unique fixed point $t\in \mathbb{T}$.
\item If $\mathfrak{m}$ hyperbolic then $\mathfrak{m}$ has one attracting fixed point $t_1\in \mathbb{T}$ with associated multiplier $\mathfrak{p}\in \mathbb{D}$ and one repelling fixed point $t_2\in \mathbb{T}$ with multiplier $\frac{1}{\mathfrak{p}}$.
\end{itemize}
\end{proposition}
\begin{proof}
See, for instance, \cite[Chapter 1]{BearBook} or \cite[Chapter 1]{MilBook}.
\end{proof}
Proposition \ref{prop:conjfixdescr} gives us a complete description of the possible dynamics of a bidegree $(1,n)$ RISPs on each fiber $F_{\lambda}$. In the next sections, we investigate how these dynamical properties vary along fibers.

\subsection{Basic notions from complex dynamics}
Our main results can be formulated in terms of the dynamics of M\"obius transformations, as discussed in the preceding Subsection \ref{subsec:Mob}. We have included some comments referencing the {\it Julia set} of a rational map $R$ on the Riemann sphere, as well as some related notions. As these comments should be viewed as short asides, we do not include full definitions, and refer the reader to \cite{BearBook,CGBook,MilBook} for background on rational iteration.

\section{RISP dynamics on the boundary}\label{sec:main}
\subsection{General remarks on boundary fixed points}
Consider a bidegree $(n_1,n_2)$ rational inner mapping $\Phi=(\phi_1, \phi_2)$, where $\phi_j=e^{i\alpha_j}\frac{\tilde{p}_j}{p_j}$, $j=1,2$, and where as usual each $p_j \in \mathbb{C}[z_1,z_2]$ has no zeros in $\mathbb{D}$. The set of fixed points of $\Phi$ in $\cc{\mathbb{D}^2}$ consists of all $z \in \overline{\mathbb{D}^2}$ such that
\[\Phi(z_1,z_2)=(z_1,z_2);\]
recall that $\Phi$ is well-defined at each $\eta \in \mathbb{T}^2$ \cite[Corollary 14.6]{Kne15}. 

We first record some facts about fixed points in the general case. For $p_j$ and $\tilde{p}_j$ as above, we define the auxiliary polynomials
\begin{equation}
P_j(z)=e^{i\alpha_j}\tilde{p}_j(z)-z_1p_j(z), \quad j=1,2,
\label{eq:fppolydef}
\end{equation}
 and we set $\Gamma(\Phi)=\mathcal{Z}(P_1)\cap \mathcal{Z}(P_2)\cap \cc{\mathbb{D}^2}$.  Then $\Gamma(\Phi)$ is comprised of the fixed points of $\Phi$ along with all singular points of $\Phi$. As before, we adopt the convention that at least one of these is made into an SF-point by renormalizing the component maps $\phi_j$, $j=1,2$.  
Note that 
\[\operatorname{deg} P_1=(n_{1}+1, n_{2}) \quad  \textrm{and} \quad \operatorname{deg} P_2=(n_1,n_2+1).\] 

The polynomials $P_j$ belong to a special class of polynomials that is often singled out in function theory in polydisks. We follow \cite{AMS06} and use the following definition.
\begin{definition}
We say that $P\in \mathbb{C}[z_1,z_2]$ is {\it essentially $\mathbb{T}^2$-symmetric} if $\tilde{P}=e^{i\beta}P$ for some $\beta \in \mathbb{R}$.
\end{definition}
We then have the following.
\begin{lemma}
If $\Phi$ is a normalized RIM, then the associated polynomials $P_1,P_2$ defined in \eqref{eq:fppolydef} are essentially $\mathbb{T}^2$-symmetric.
\end{lemma}
\begin{proof}
This amounts to a computation. Applying the reflection operation \eqref{eq:preflection} to the polynomial $P_1$ and exploiting the fact that reflection is an involution, we obtain
\begin{multline*}
\tilde{P}_1(z_1,z_2)=z_1^{n_1+1}z_2^{n_2}e^{-i\alpha_1}\cc{\tilde{p}}_1\left(\frac{1}{\cc{z}_1}, \frac{1}{\cc{z}_2}\right)-z_1^{n_1+1}z_2^{n_2}\frac{1}{z_1}\cc{p}_1\left(\frac{1}{\cc{z}_1}, \frac{1}{\cc{z}_2}\right)\\
=e^{-i\alpha_1}z_1p_1(z_1,z_2)-\tilde{p}_1(z_1,z_2)=-e^{-i\alpha_1}(-z_1p_1(z_1,z_2)+e^{i\alpha_1}\tilde{p}_1(z_1,z_2))\\=-e^{-i\alpha_1}P_1.
\end{multline*}

A similar computation establishes the assertion for $P_2$.
\end{proof}
Essentially $\mathbb{T}^2$-symmetric polynomials have been studied extensively, see for instance \cite{AMS06,Kne10,BPS20} and the references therein. In particular, such polynomials are related to determinantal representations and so-called distinguished varieties. We caution that the focus is often on essentially $\mathbb{T}^2$-symmetric polynomials with no zeros in the bidisk, while the polynomials $P_j$ that we encounter in this work typically {\it do} have some zeros in $\mathbb{D}^2$, namely at interior fixed points of $\Phi$. As a consequence, some of the known results from the literature (e.g. \cite{BKPSprep}) no longer apply.

In the pictures in Section \ref{sec:Ex}, we observed the presence of curves of fixed points in the $2$-torus. The fixed point set of a general RIM, however, typically consists of isolated points. This can be seen via a standard application of B\'ezout's theorem. (See \cite[Section 1.3]{SibSurv} for applications to counting the number of fixed and periodic points in the context of iteration of rational functions in $\mathbb{P}^k$.)
\begin{lemma}\label{lem:finitefixpts}
Suppose $\Phi=(\phi_1, \phi_2)$ is a RIM such that the associated polynomials $P_1$ and $P_2$ have no common factors.
Then $\Gamma(\Phi)$ is a finite set, with cardinality bounded by $2n_1n_2+n_1+n_2$.
\end{lemma}
\begin{proof}
Lemma \ref{lem:finitefixpts} follows directly from B\'ezout's theorem as discussed in Section \ref{sec:Prelim}. Namely, since $P_1$ and $P_2$ have no common zeros, we have
\[\#(\mathcal{Z}(P_1)\cap \mathcal{Z}(P_2))=(n_1+1)\cdot(n_2+1) +n_2\cdot n_1=2n_1n_2+n_1+n_2.\]
This is an upper bound on the number of elements of $\Gamma(\Phi)$ since some common zeros could be located in $(\overline{\mathbb{D}^2})^c$.
\end{proof}
Note that the assumption that $P_1$ and $P_2$ have no common factor is important: if we set  $\phi_1$ and $\phi_2$ both equal to the RIF in Example \ref{ex:fave}, then $P_1$ and $P_2$ are both divisible by $z_1-z_2$, and the diagonal $\{z_1=z_2\}$ is fixed by the resulting mapping $\Phi$. 

We now focus on fixed points of a RISP $\Phi=(\phi,z_2)$, where $\phi=\frac{\tilde{p}}{p}$, and $p$ has no zeros in $\mathbb{D}^2$.  Note that for such functions, the polynomial $P_2$ in \eqref{eq:fppolydef} becomes identically 0, and so  Lemma \ref{lem:finitefixpts} does not apply. Instead, in this case, we get that $z \in \mathbb{T}^2$ satisfies the single condition 
\begin{equation}
q(z)-z_1p(z)=0,
\label{eq:singlefix}
\end{equation}
and we can expect to see curves of fixed points in $\mathbb{T}^2$. Our goal is to investigate the nature of these fixed point curves.

We begin with a simple consequence of work in \cite{BPS20}.
\begin{proposition}\label{prop:BPSconseq}
Let $\Phi=(z_1\phi, z_2)$, where $\phi$ is a normalized RIF. Then $\Gamma(\Phi)\cap \mathbb{T}^2$ is a union of curves, and each component of $\Gamma(\Phi)\cap \mathbb{T}^2$ can be parametrized by analytic functions.
\end{proposition}
The subtlety here is that $\Phi$ may have singularities on $\mathbb{T}^2$, so that the implicit function theorem does not apply directly.
\begin{proof}
By assumption, the first coordinate of $\Phi$ can be written $$z_1 \phi = z_1e^{i\alpha_1}\frac{\tilde{p}(z_1,z_2)}{p(z_1,z_2)}$$ for some real $\alpha_1$ and some polynomial $p$ with no zeros on $\mathbb{D}^2$. Then, the fixed point condition \eqref{eq:singlefix} takes on the form
\[0=\eta_1e^{i\alpha}\tilde{p}(\eta)-\eta_1p(\eta)=\eta_1(e^{i\alpha}\tilde{p}(\eta)-p(\eta)).\]
Since $\eta_1\neq 0$ for $\eta_1\in \mathbb{T}$, we get that $\eta\in \mathbb{T}^2$ is a fixed point of $\Phi$ if and only if $e^{i\alpha}\tilde{p}(\eta)-p(\eta)=0$. This latter condition  means that $\eta$ belongs to some unimodular level set of the RIF $\phi$. Appealing to \cite[Theorem 2.8]{BPS20}, we arrive at the desired conclusion.
\end{proof}
In particular, Proposition \ref{prop:BPSconseq} implies that RISPs of the particular form $(z_1\phi, z_2)$ exhibit an abundance of fixed points in the $2$-torus. On each fiber, the function $z_1\mapsto z_1\phi_{\lambda}(z_1)$ is a Blaschke product of degree at least $1$, and for all but finitely many values of $\lambda$, the degree is at least $2$. This means that the dynamical behavior on a $\lambda$-fiber generically does not admit a description in elementary terms. 

For instance, $z_1\phi_{\lambda}(z_1)$ has a fixed point at $z_1=0$ for each choice of $\lambda$ and this fixed point is attracting whenever $z_1\phi_{\lambda}$ has degree at least $2$. Moreover since for generic $\lambda\in \mathbb{T}$, the numerator of $z_1\phi_{\lambda}(z_1)$ has higher degree than the denominator, the point at infinity of $\mathbb{C}_{\infty}$ is also an attracting fixed point. This implies \cite[Chapter 7]{MilBook} that for generic $\lambda$, the Julia set of the rational map $z_1\phi_{\lambda}(z_1)\colon \mathbb{C}_{\infty}\to \mathbb{C}_{\infty}$ is $\mathbb{T}$, the unit circle. We thus have complicated dynamical behavior on each fiber $F_{\lambda}$ and (see \cite[Chapter 6]{BearBook} or \cite[Chapter 4]{MilBook}), and for all but finitely many values of $\lambda\in \mathbb{T}$, the points in $F_{\lambda}\cap \Gamma(\phi)$ are repelling fixed points of the rational map $z_1\phi_{\lambda}(z_1)$. What Proposition \ref{prop:BPSconseq} asserts is that these repelling fixed points move in an analytic fashion in the fiber variable $\lambda$. 

The following example illustrates some of the facts listed above.
\begin{ex}
Consider the RISP
\[\Phi(z)=\left(-z_1\frac{2z_1z_2-z_1-z_2}{2-z_1-z_2}, z_2\right).\] 
As was noted above, the set $\{(0,z_2)\}\subset \cc{\mathbb{D}^2}$ consists of fixed points of $\Phi$. On the fiber $z_2=1$, $\Phi$ reduces to the identity, and on all other $F_{\lambda}$, the fiber map $z_1\phi_{\lambda}(z_1)$ has degree $2$ and has attracting fixed points at $0$ and $\infty$. 

Solving $-(2z_1z_2-z_1-z_2)=2-z_1-z_2$, we find that $\{(1, \lambda)\colon \lambda \in \mathbb{T}\setminus \{1\}\}\subset \mathbb{T}^2$ is fixed by $\Phi$, so that $1\in \mathbb{T}$ is a fixed point on each $F_{\lambda}$. A direct computation shows that $\frac{d}{dz_1}(-z_1\phi_{\lambda})(1)=3$ when $\lambda\neq 1$, confirming the presence of a repelling fixed point.

The dynamics of $\Phi$ are visualized in Figure \ref{fig:favejulia} and indeed have a more complicated appearance than the examples in Section \ref{sec:Ex}.

\begin{figure}[h!]
    \subfigure[Vertical lines before mapping.]
      {\includegraphics[width=0.4 \textwidth]{vertparacurves.pdf}}
    \hfill
    \subfigure[$\Phi^n$ for $n=1$.]
      {\includegraphics[width=0.4 \textwidth]{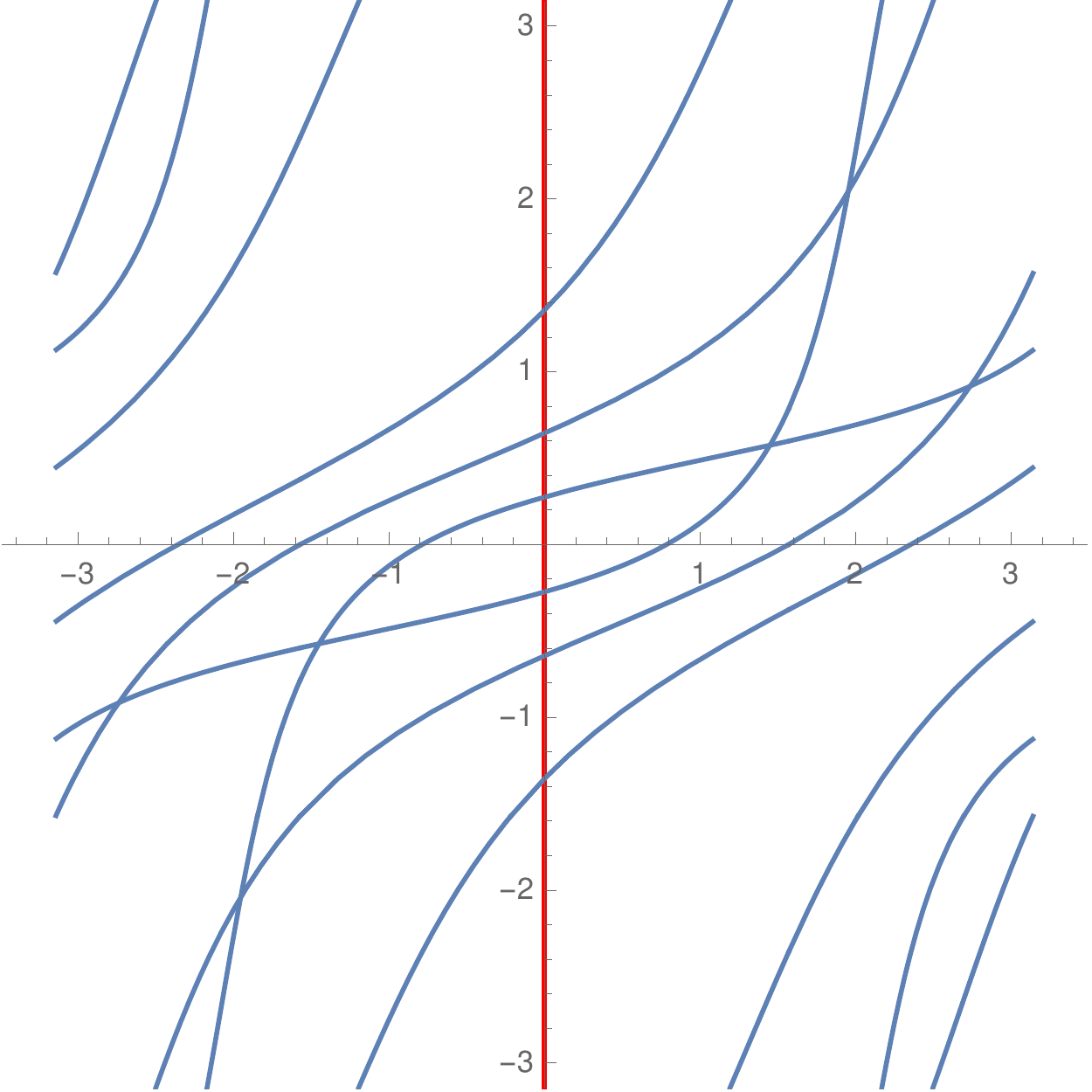}}
      \subfigure[$\Phi^n$ for $n=2$]
      {\includegraphics[width=0.4 \textwidth]{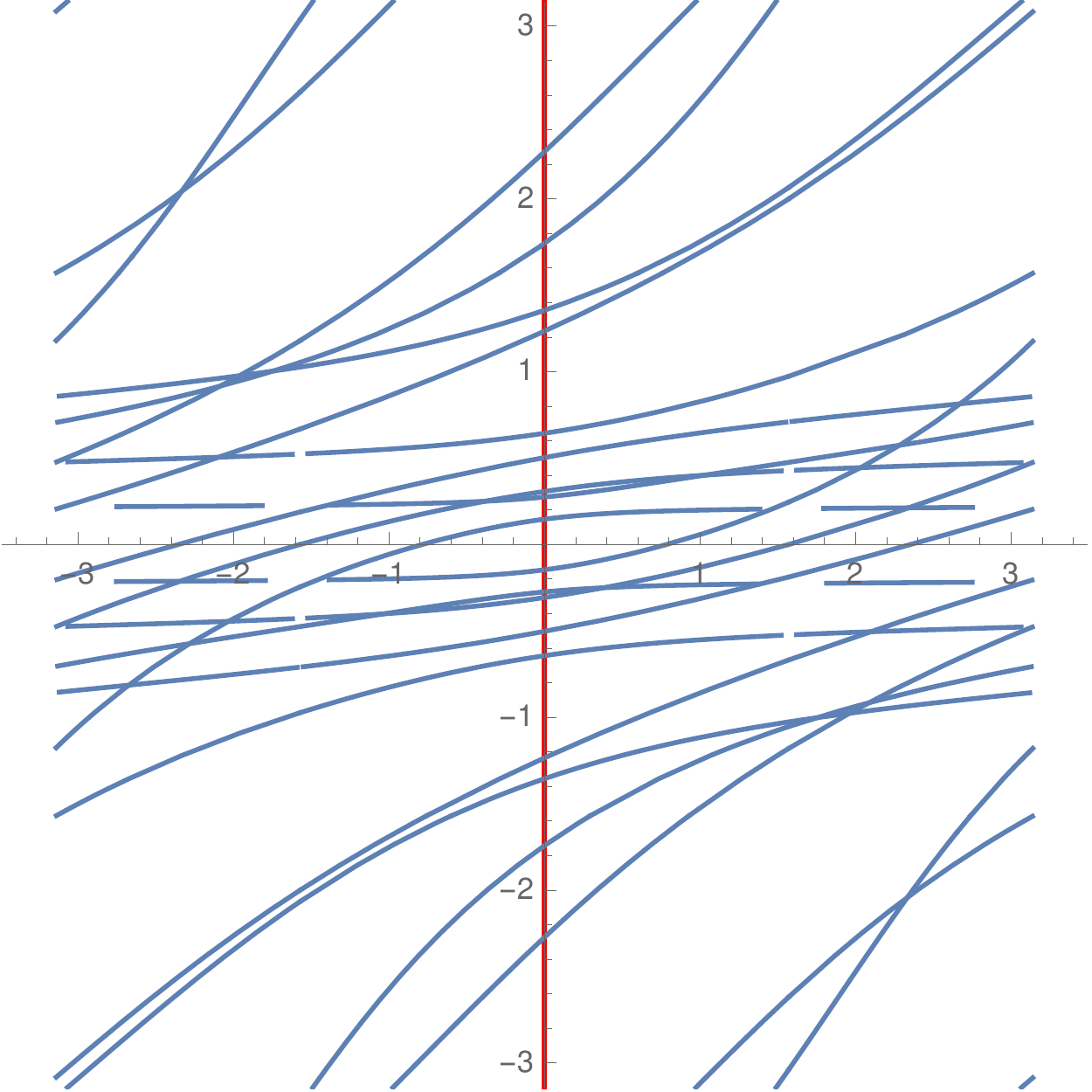}}
    \hfill
    \subfigure[$\Phi^n$ for $n=5$.]
      {\includegraphics[width=0.4 \textwidth]{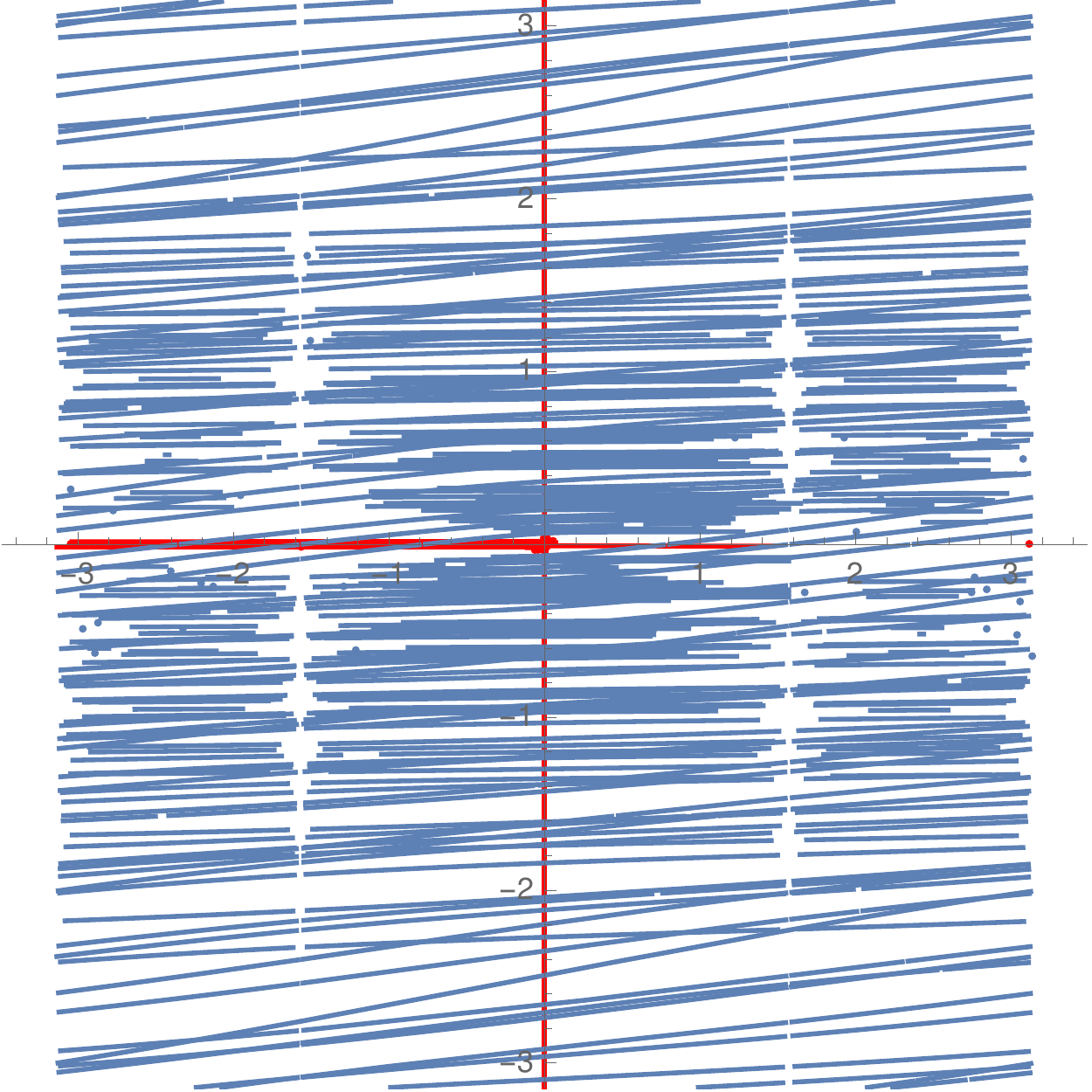}}
  \caption{\textsl{Iteration of $\Phi=(-z_1\frac{2z_1z_2-z_1-z_2}{2-z_1-z_2},z_2)$ on $\mathbb{T}^2$.}}
  \label{fig:favejulia}
  
\end{figure}
\end{ex}

\subsection{Rotation belts and parabolic fixed points for simple RISPs}
We return our focus to the simple case where $\Phi=(\phi, z_2)$ and $\phi$ is a bidegree $(1,n)$ normalized RIF, and where the fiber dynamics are completely described by Proposition \ref{prop:conjfixdescr}. In this case, since $p$ and $\tilde{p}$ are polynomials of degree $1$ in the first variable, \eqref{eq:singlefix} reduces to the quadratic equation
\[P(z_1,z_2)=e^{i\alpha}\tilde{p}(z_1,z_2)-z_1p(z_1,z_2)=0.\]
If $\Phi$ has a SF-point in $\mathbb{T}^2$ then $\mathcal{Z}(P)\cap \mathbb{T}^2\neq \emptyset$ but it may well happen that $\mathcal{Z}(P)\cap (F_{\lambda}\cap \mathbb{T}^2)=\emptyset$ for some values of $\lambda \in \mathbb{T}$. With this in mind, we make the following definition.
\begin{definition}\label{def:rotbelt}
Let $\lambda_1, \lambda_2\in [-\pi, \pi]$ and suppose $\lambda_1<\lambda_2$. We say that the set
\[B(\lambda_1, \lambda_2)=\mathbb{T}\times \{e^{it}\colon \lambda_1<t<\lambda_2 \}\subset \mathbb{T}^2\]
is a {\it rotation belt} for $\Phi=(\phi, z_2)$ if 
\begin{itemize}
\item no point belonging to $B(\lambda_1,\lambda_2)$ is a fixed point of $\Phi$
\item each of the sets $F_{\lambda_1}\cap \mathbb{T}^2$ and $F_{\lambda_2}\cap \mathbb{T}^2$ contains at least one fixed point of $\Phi$.
\end{itemize}
\end{definition}
The reason for our terminology here is that by Proposition \ref{prop:conjfixdescr}, each fiber map $\phi_{\lambda}$ associated with fibers contained in $B(\lambda_1,\lambda_2)$ is conjugate to a rotation. The second assumption is essentially a maximality condition. Examples \ref{ex:fave} and \ref{ex:para} illustrate that not all RISPs possess rotation belts. Example \ref{ex:rot} has a single rotation belt, but in general several rotation belts may be present. 

We now give a criterion for a degree $(1,n)$ RISP to possess parabolic fixed points in terms of the one-variable polynomial $Q_{\alpha}$. As in \eqref{def:p1p2polys} in Section \ref{sec:Prelim}, we write $\phi=e^{i\alpha}\frac{\tilde{p}}{p}$, where
\[p(z)=p_1(z_2)+z_1p_2(z_2), \]
and $\tilde{p}(z)=\tilde{p}_2(z_2)+z_1\tilde{p}_1(z_2)$ for a pair of one-variable polynomials $p_1,p_2$. Recall that $$Q_{\alpha}=(p_1-e^{i\alpha}\tilde{p}_1)^2+4e^{i\alpha}p_2\tilde{p}_2.$$
We shall need the auxiliary expressions
\begin{equation}
\psi^{1}_{\alpha}(z_2)=\frac{1}{2p_2(z_2)}\left(p_1(z_2)-e^{i\alpha}\tilde{p}_1(z_2)+\sqrt{Q_{\alpha}(z_2)}\right)
\label{psi1}
\end{equation}
and 
\begin{equation}
\psi^{2}_{\alpha}(z_2)=\frac{1}{2p_2(z_2)}\left(p_1(z_2)-e^{i\alpha}\tilde{p}_1(z_2)-\sqrt{Q_{\alpha}(z_2)}\right).
\label{psi2}
\end{equation}
The functions $\psi^j_{\alpha}$ parametrize the roots of the polynomial 
\[P_ {\alpha}(z)=e^{i\alpha}\tilde{p}-z_1p=e^{i\alpha}\tilde{p}_2+e^{i\alpha}z_1\tilde{p}_1-z_1p_1-z_1^2p_2\]
away from the set $\Lambda^{\sharp}=\{z_2\in \mathbb{C}\colon p_2(z_2)=0\}$. Note that $\Lambda^{\flat}\cap \Lambda^{\sharp}=\emptyset$ since $p_1(z_2)\neq 0$ for $z_2\in \overline{\mathbb{D}}$.

First, let us examine what happens on the set $\Lambda^{\sharp}$.
\begin{lemma}
Suppose $\lambda \in \mathcal{Z}(Q_{\alpha})\cap (\Lambda^{\sharp}\cap\mathbb{T})$. Then $\phi_{\lambda}(z_1)=e^{-i\alpha}z_1$.

Conversely, if $\phi_{\lambda}(z_1)=e^{-i\alpha}z_1$ for some $\lambda \in \mathbb{T}$, then $p_2(\lambda)=0$ and 
$\lambda \in \mathcal{Z}(Q_{\alpha})$.
\end{lemma}
\begin{proof}
By the definition of the reflection operation, $p_2(\lambda)=0$ implies $\tilde{p}_2(\lambda)=0$. Next, if also $Q_{\alpha}(\lambda)=0$ then $(p_1(\lambda)-e^{i\alpha}\tilde{p}_1(\lambda))^2=0$ so that $p_1( \lambda)=e^{i\alpha}\tilde{p}_1(\lambda)$. 
But then
\[\phi(z_1,\lambda)=\frac{z_1\tilde{p}_1(\lambda)+0}{p_1(\lambda)+0}=e^{-i\alpha}z_1.\]

Conversely, suppose
\[e^{-i\alpha}z_1=\frac{z_1\tilde{p}_1(\lambda)+\tilde{p}_2(\lambda)}{p_1(\lambda)+p_2(\lambda)z_1}\]
for some unimodular $\lambda$. Then, comparing coefficients, we get $p_2(\lambda)=0$ and as a direct consequence that $\tilde{p}_2(\lambda)=0$. 
Finally, we read off that $e^{-i\alpha}=\frac{\tilde{p}_1(\lambda)}{p_1(\lambda)}$ and hence $Q_{\alpha}(\lambda)=0$, as claimed.
\end{proof}
Having seen what happens on the set $\Lambda^{\sharp}$, we now examine the behavior of $Q_{\alpha}$ on $\Lambda^{\flat}$.
\begin{lemma}\label{lem:Qvanishing}
Suppose $\lambda \in \Lambda^{\flat}$. Then 
\begin{enumerate}
\item $Q_{\alpha}(\lambda)=0$ if and only if $p_1(\lambda)+e^{i\alpha}\tilde{p}_1(\lambda)=0$;
\item if $Q_{\alpha}(\lambda)=0$ then $Q_{\alpha}$ vanishes to at least order $2$ at $\lambda$.
\end{enumerate}
\end{lemma}
\begin{proof}
We first prove (1). By completing the square, we can rewrite $Q_{\alpha}$ as
\[Q_{\alpha}(z_2)=(p_1(z_2)+e^{i\alpha}\tilde{p}_1(z_2))^2+4e^{i\alpha}(p_2(z_2)\tilde{p}_2(z_2)-p_1(z_2)\tilde{p}_1(z_2)).\]
Since $\lambda \in \Lambda^{\flat}$, Lemma \ref{lem:collapsefiberlem} implies that $p(\tau_1,\lambda)=\tilde{p}(\tau_1,\lambda)=0$ for some $\tau_1\in \mathbb{T}$. Solving for $\tau_1$, we deduce that $\frac{p_1(\lambda)}{p_2(\lambda)}=\frac{\tilde{p}_2(\lambda)}{\tilde{p}_1(\lambda)}$ and then $p_1(\lambda)\tilde{p}_1(\lambda)-p_2(\lambda)\tilde{p}_2(\lambda)=0$.

Hence, $Q_{\alpha}(\lambda)=(p_1(\lambda)+e^{i\alpha}\tilde{p}_1(\lambda))^2$ leading to the desired condition.

We turn to (2). Clearly, any zero of $(p_1+e^{i\alpha}\tilde{p}_1)^2$ occurs with even multiplicity. Next, we note that $\left|\frac{p_2(z_2)}{p_1(z_2) }\right|^2\leq 1$ for $z_2\in \mathbb{D}$, for otherwise $p$ would have a zero in $\mathbb{D}^2$. Now, we note that
\[p_2(e^{it_2})\tilde{p}_2(e^{it_2})-p_1(e^{it_2})\tilde{p}_1(e^{it_2})=e^{int_2}\left(|p_2(e^{it_2})|^2-|p_1(e^{it_2})|^2\right)\]
and since the expression on the right is a non-positive trigonometric polynomial, the Fej\'er-Riesz theorem implies that it is equal to $-|r(e^{it_2})|^2$ for some polynomial $r$, with a zero at $\lambda$ by the proof of (1).
\end{proof}
We shall see later, in Example  \ref{ex:multisingex}, that it may well happen that $Q_{\alpha}(\lambda)\neq 0$ for certain $\lambda \in \Lambda^{\flat}$.

\begin{theorem}\label{thm:paracrit}
Suppose $Q_{\alpha}(\lambda)=0$ for some $\lambda \in \mathbb{T}\setminus (\Lambda^{\flat}\cup \Lambda^{\sharp})$. Then $\psi^{1}(\lambda)$ and $\psi^{2}(\lambda)$ coincide as an element of $\mathbb{T}$, and $\phi_{\lambda}$ has a parabolic fixed point at $\psi^1_{\alpha}(\lambda)$. 
\end{theorem}
Note that $\psi^{j}_{\alpha}$, $j=1,2$, will in general exhibit branch points at points where $Q_{\alpha}$ vanishes, and so $\Gamma(\Phi)$ does not in general admit a description in terms of analytic functions near parabolic fixed points of $\Phi$ in $\mathbb{T}^2$. This is in contrast to the results in \cite{BPS20} which state that level sets of RIFs can be parametrized in such a way at every point of $\mathbb{T}^2$. 
\begin{proof}
Because of our assumptions, for each $\lambda \in \mathbb{T} \setminus (\Lambda^{\flat}\cup \Lambda^{\sharp})$, the fiber map $\phi_{\lambda}$ is a M\"obius transformation not of the form $e^{i\beta}z_1$. Then, by the discussion in Section \ref{sec:Prelim}, $\phi_{\lambda}$ either has two hyperbolic fixed points on the unit circle, or a parabolic fixed point on $\mathbb{T}$, or else an elliptic fixed point in $\mathbb{D}$ with another fixed point lying in $\overline{\mathbb{D}^c}$. But $Q_{\alpha}(\lambda)=0$ with 
$\lambda \notin \Lambda^{\sharp}$ implies that $\psi^{1}_{\alpha}(\lambda)=\psi^{2}_{\alpha}(\lambda)$. Since $\lambda\notin \Lambda^{\flat}$, we are in the parabolic case and then necessarily $\psi^{1}_{\alpha}(\lambda) \in \mathbb{T}$.
\end{proof}
\begin{ex}
Consider the polynomials $p$ and $\tilde{p}$ in appearing in the RIF $\phi$ in Example \ref{ex:rot} in Section \ref{sec:Ex}. There, $p_1(z_2)=3-z_2$ and $p_2(z_2)=-1$, while 
$\tilde{p}_1(z_2)=3z_2-1$ and $\tilde{p}_2(z_2)=-z_2$. Then $Q_{\alpha}(z_2)=16z_2^2-28z_2+16$, and this polynomial has conjugate roots on $\lambda_{1,2}=\frac{1}{8}(7\pm \sqrt{15}i)\in\mathbb{T}$. The corresponding fibers $F_{\lambda_{1,2}}$ bound the rotation belt of $\Phi$.
\end{ex}
\begin{ex}
It may well happen that $Q_{\alpha}\equiv 0$, so that most fixed points are parabolic. This can be seen in Example \ref{ex:para} from Section \ref{sec:Ex}. 
In that case, we have $p_1(z_2)=3-z_2$ and $p_2(z_2)=-(1+z_2)$ as well as $\tilde{p}_1(z_2)=3z_2-1$ and $\tilde{p}_2(z_2)=-(1+z_2)$, and if we insist on the normalized choice $\alpha=\pi$, then indeed $Q_{\alpha}$ is the zero polynomial. As we saw before, the set $\{(1,e^{it_2})\colon t_2\neq 0,\pi\}$ consists of parabolic fixed points of $\Phi$.

This example also illustrates the need to assume $\lambda \notin \Lambda^{\flat}\cup \Lambda^{\sharp}$. Namely, $\phi(z_1,1)=1$ for $\lambda=1\in \Lambda^{\flat}$ and $\phi(z_1,-1)=-z_1$ for $\lambda=-1 \in \Lambda^{\sharp}$, and neither case leads to a parabolic fixed point.
\end{ex}
\begin{theorem}\label{thm:beltcount}
Suppose $\Phi$ is a degree $(1,n)$ RISP with associated $Q_{\alpha}$ not identically zero. Then the number of rotation belts for $\Phi$ is less than or equal to
\[\frac{1}{2}\#\left[(\mathcal{Z}(Q_{\alpha}) \setminus \Lambda^{\flat})\cap \mathbb{T}\right].\]
\end{theorem}
\begin{proof}
Note that $\Gamma(\Phi)$ is parametrized by the analytic functions $\psi^{1,2}_{\alpha}$ away from the branch points, and that the fiber map derivative $\frac{d}{dz_1}\phi_{\lambda}$ is continuous off $\iota(\Lambda^{\flat})$, see Lemma \ref{lem:contlemma} below. Thus, within a rotation belt $B(\lambda_1, \lambda_2)$ the composite map $\lambda \mapsto \mathfrak{p}(\lambda) \in \mathbb{T}$ is continuous. Now if a curve component of $\Gamma(\Phi)$ containing elliptic fixed points were to meet $\mathbb{T}^2$ in a hyperbolic fixed point having $|\mathfrak{p}(\lambda)|\neq 1$, this would force a discontinuity in the multiplier at $\lambda_1$ or $\lambda_2$. Hence, the boundary fiber of a rotation belt contains a parabolic point, or is contained in $\iota(\Lambda^{\flat}\cup \Lambda^{\sharp})$. 

Next, we observe that neither of the boundary fibers $F_{\lambda_1}, F_{\lambda_2}$ can be collapsing. For if, say, $F_{\lambda_1}$ was collapsing, then it could contain at most finitely many singularites $(\tau_1,\lambda_1), \ldots, (\tau_m, \lambda)$ of $\phi$ by Lemma \ref{lem:collapsefiberlem}. But the collapsing of $F_{\lambda_1}$ would imply that $\Phi$ had discontinuities as $\lambda \to \lambda_1$ in $B(\lambda_1,\lambda_2)$ along vertical lines $\{\tau\}\times \mathbb{T}\subset \mathbb{T}^2$ for $\tau\neq \tau_1, \ldots, \tau_m$, which is impossible. Hence, by Theorem \ref{thm:paracrit}, this implies that the fibers $F_ {\lambda_1}$ and $F_{\lambda_2}$ associated with a rotation belt $B(\lambda_1,\lambda_2)$ contain a parabolic fixed point, or, if $\lambda \in \mathcal{Z}(Q_{\alpha})\cap \Lambda^{\sharp}$ and $\alpha=0$, have fiber map $\phi_{\lambda}=z_1$.
\end{proof}
Note that $\deg Q_{\alpha} \leq 2n$, so that $Q_{\alpha}$ has at most $2n$ zeros on $\mathbb{T}$ counting multiplicities. It may well happen, however, that $\Phi$ has fewer than $n$ rotation belts if $Q_{\alpha}(\lambda_j)=0$ for some $\lambda_j\in \Lambda^{\flat}$; see Section \ref{sec:morex} for an example with $n=2$ and a single rotation belt. Similarly, examples show that the situation where a rotation belt is bounded by a fiber where the fiber map is equal to the identity can indeed arise.

\begin{lemma}\label{lem:contlemma}
Let $\mathfrak{p}(\lambda)$ denote the multiplier of $\phi_{\lambda}$ at a point $\lambda \in \overline{\mathbb{D}}$. Then $\lambda\mapsto \mathfrak{p}(\lambda)$ is an analytic function on each component of $\Gamma(\phi)\setminus [\iota(\mathcal{Z}(Q_{\alpha}))\cap\iota(\Lambda^{\flat}\cup \Lambda^{\sharp})]$.
\end{lemma}
\begin{proof}
Note that for each $\lambda \in \mathbb{T}\setminus (\Lambda^{\flat}\cup \Lambda^{\sharp})$, the M\"obius map $\frac{d}{dz_1} \phi_ {\lambda}$ is non-trivial and has no critical points. This implies that $\lambda  \mapsto \frac{d}{dz_1} \phi_ {\lambda}(\psi^{j}_{\alpha}(\lambda))$ is analytic off $\mathcal{Z}(Q_{\alpha})$ and $\Lambda^{\flat}\cup \Lambda^{\sharp}$.
\end{proof}
Lemma \ref{lem:contlemma} has the following consequence for rotations belts.
\begin{lemma} 
Suppose $B(\lambda_1, \lambda_2)$ is a rotation belt for $\Phi=(\phi,z_2)$ and that $\mathfrak{p}(\lambda)$ is non-constant on $(\lambda_1, \lambda_2)$. Then uncountably many $\phi_{\lambda}$ are conjugate to an irrational rotation.
\end{lemma}
\begin{proof}
By assumption $-i\log(\mathfrak{p}(\lambda))$ is a non-constant smooth function taking values in $[-\pi, \pi]$, and since moreover $-i\log(\mathfrak{p}(\lambda_{1,2}))=0$, the statement follows from the intermediate value theorem.
\end{proof}

\subsection{Hyperbolic fixed points and SF-points for simple RISPs}
We now examine the dynamical behavior of $\Phi$ in the vicinity of an SF-point. 

All singularities of a RISP are points of $\mathbb{T}^2$, and lie in collapsing fibers by Lemma \ref{lem:collapsefiberlem}. 
\begin{lemma}
Let $(\tau_1, \lambda_1), \ldots, (\tau_n, \lambda_m)$ be SF-points of $\Phi$. Then $(\tau_j, \lambda_j) \in \Gamma(\Phi)\cap \mathbb{T}^2$ for $j=1,\ldots, m$.
\end{lemma} 
The proof of Theorem \ref{thm:beltcount} shows that a collapsing fiber cannot be contained in a rotation belt $B(\lambda_1, \lambda_2)$, nor can it coincide with one of the boundary fibers $F_{\lambda_1}$ or $F_{\lambda_2}$. We have already seen that SF-points can be accessed in $\mathbb{T}^2$ via a component of $\Gamma(\Phi) \cap \mathbb{T}^2$ in different ways. An SF-point may belong to a component of $\Gamma(\Phi)\cap \mathbb{T}^2$ consisting of parabolic points (Example \ref{ex:para}), or to one or more components of $\Gamma(\Phi) \cap \mathbb{T}^2$ made up of hyperbolic fixed points (Example \ref{ex:fave}).

\begin{theorem}
Let $(\tau, \lambda) \in \mathbb{T}^2$ be an SF-point of a RISP $\Phi$ with associated polynomial $Q_{\alpha}$ not identically zero.
\begin{enumerate}
\item If $Q_{\alpha}(\lambda)\neq 0$ then $(\tau, \lambda)$ belongs to a single component of $\Gamma(\Phi) \cap \mathbb{T}^2$. Moreover, this component can be parametrized by an analytic function in some neighborhood of $(\tau, \lambda)$.
\item If $Q_{\alpha}(\lambda)=0$  then $(\tau, \lambda)$ belongs to two components of $\Gamma(\Phi)\cap \mathbb{T}^2$. If $Q_{\alpha}$ vanishes to even order, then each of these components can be parametrized by an analytic function in some neighborhood of $(\tau, \lambda)$.
\end{enumerate}
\end{theorem}
\begin{proof}
Item (1) follows from the definition of $\psi^{j}$, $j=1,2$, in \eqref{psi1} and \eqref{psi2} together with the observation that points in $\Lambda^{\flat}$ cannot be elements of $\Lambda^{\sharp}$.

Let us turn to item (2). Lemma \ref{lem:Qvanishing} asserts that $Q_{\alpha}$ vanishes to order at least $2$. If $Q_{\alpha}$ in fact vanishes to even order at $z_2=\lambda$, then the function $\sqrt{Q_{\alpha}(z_2)}$ is analytic in a neighborhood of $\lambda$. Hence $\psi^{j}$, $j=1,2$, are analytic also.
\end{proof}

In other words, a pair of components of $\Gamma(\Phi)$ coming together at an SF-point generically cross transversally. In fact, examination of a wide range of examples (see \cite{TDpage}) suggests that the components of $\Gamma(\Phi)$ can always be parametrized by analytic functions in a neighborhood of an SF-point, but we have not been able to find an elementary proof. 

As an example, the RISP in Example \ref{ex:multisingex} has two SF-points, with two hyperbolic fixed point curves meeting at one of them, and a single hyperbolic fixed point curve passing through the other.

\section{Further examples}\label{sec:morex}

We conclude by examining two examples with more intricate singularities and dynamical behavior. Further examples and images can be found on the webpage \cite{TDpage}. Also included on \cite{TDpage} are images associated with the two-dimensional Blaschke products considered in \cite{PS08,PR10}, and with other types of RIMs that are not studied in detail in this paper.
\begin{ex}\label{ex:AMYex}
\begin{figure}[h!]
    \subfigure[$\Phi^n$ for $n=1$.]
      {\includegraphics[width=0.35 \textwidth]{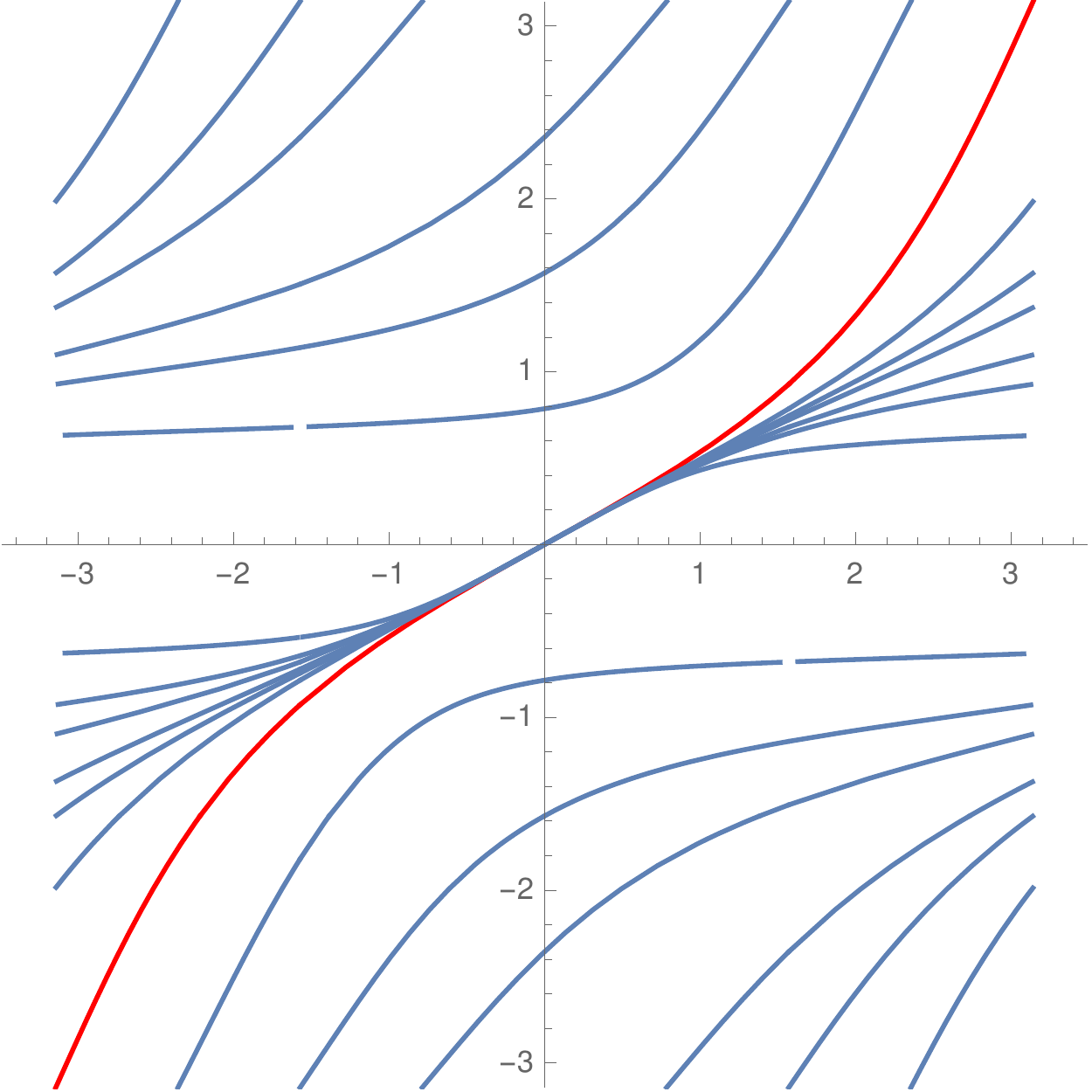}}
    \hfill
    \subfigure[$\Phi^n$ for $n=2$.]
      {\includegraphics[width=0.35 \textwidth]{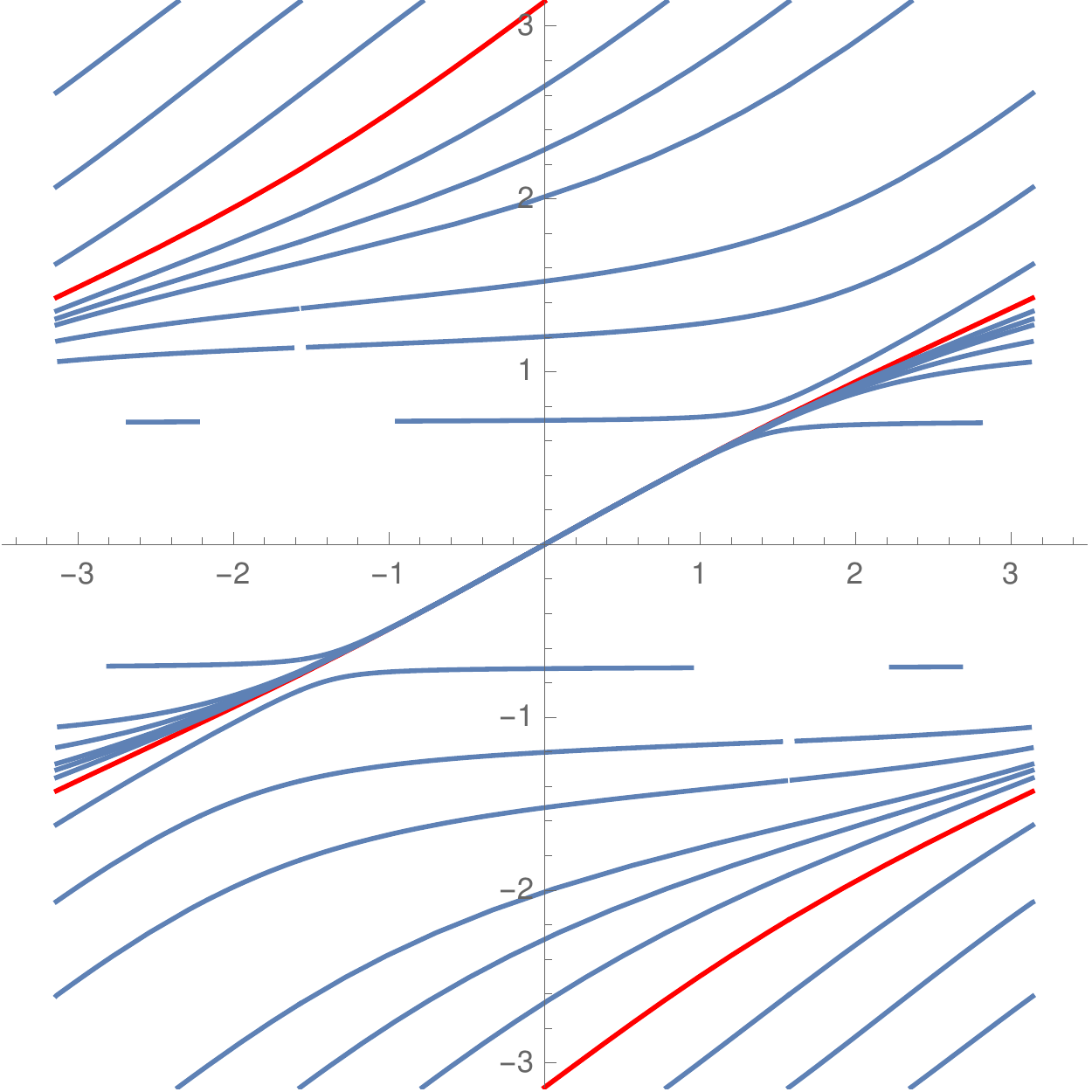}}
      \subfigure[$\Phi^n$ for $n=5$]
      {\includegraphics[width=0.35 \textwidth]{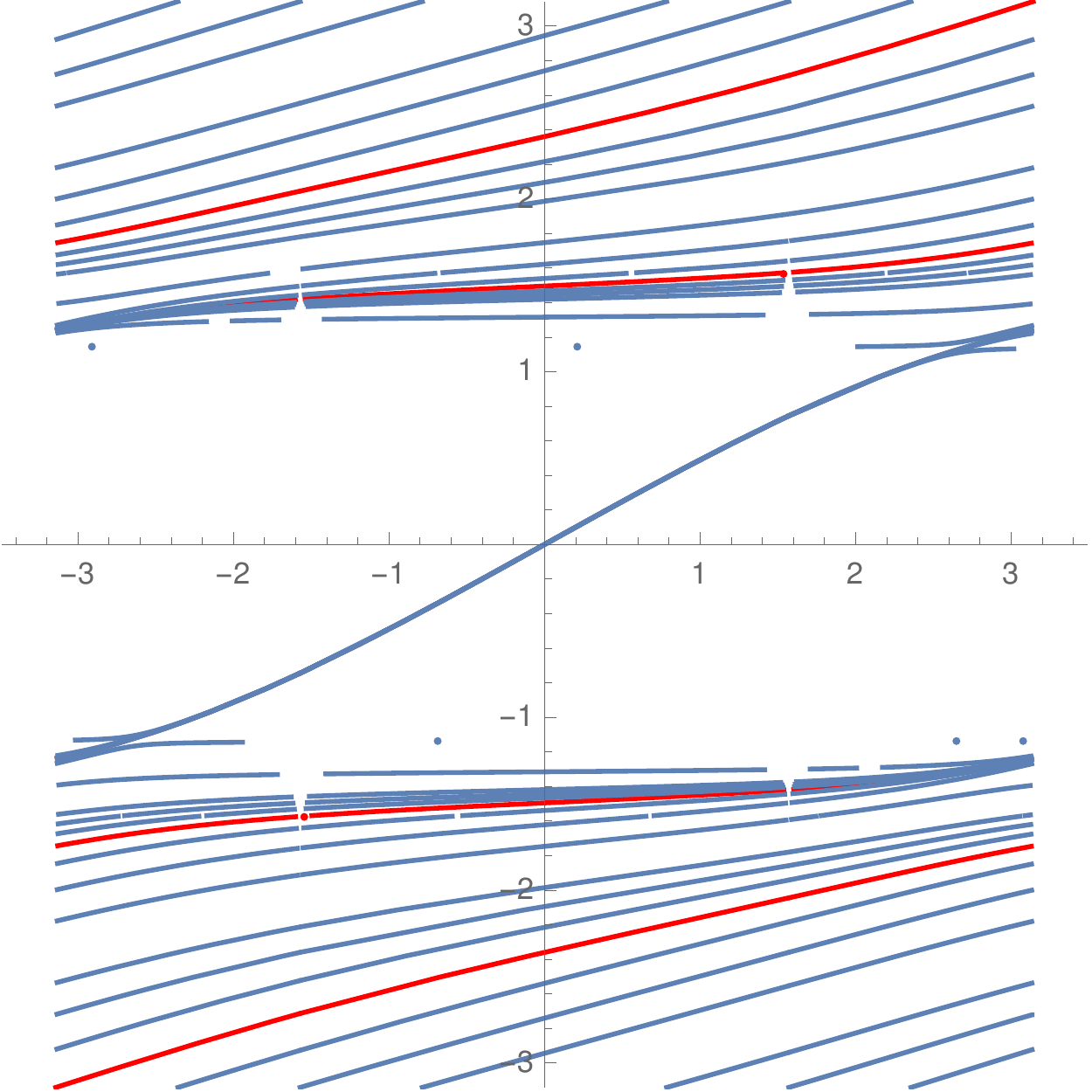}}
    \hfill
    \subfigure[$\Gamma(\Phi)\cap\mathbb{T}^2$. Parabolic fibers $F_{\lambda_{2,3}}$ in pink.]
      {\includegraphics[width=0.35 \textwidth]{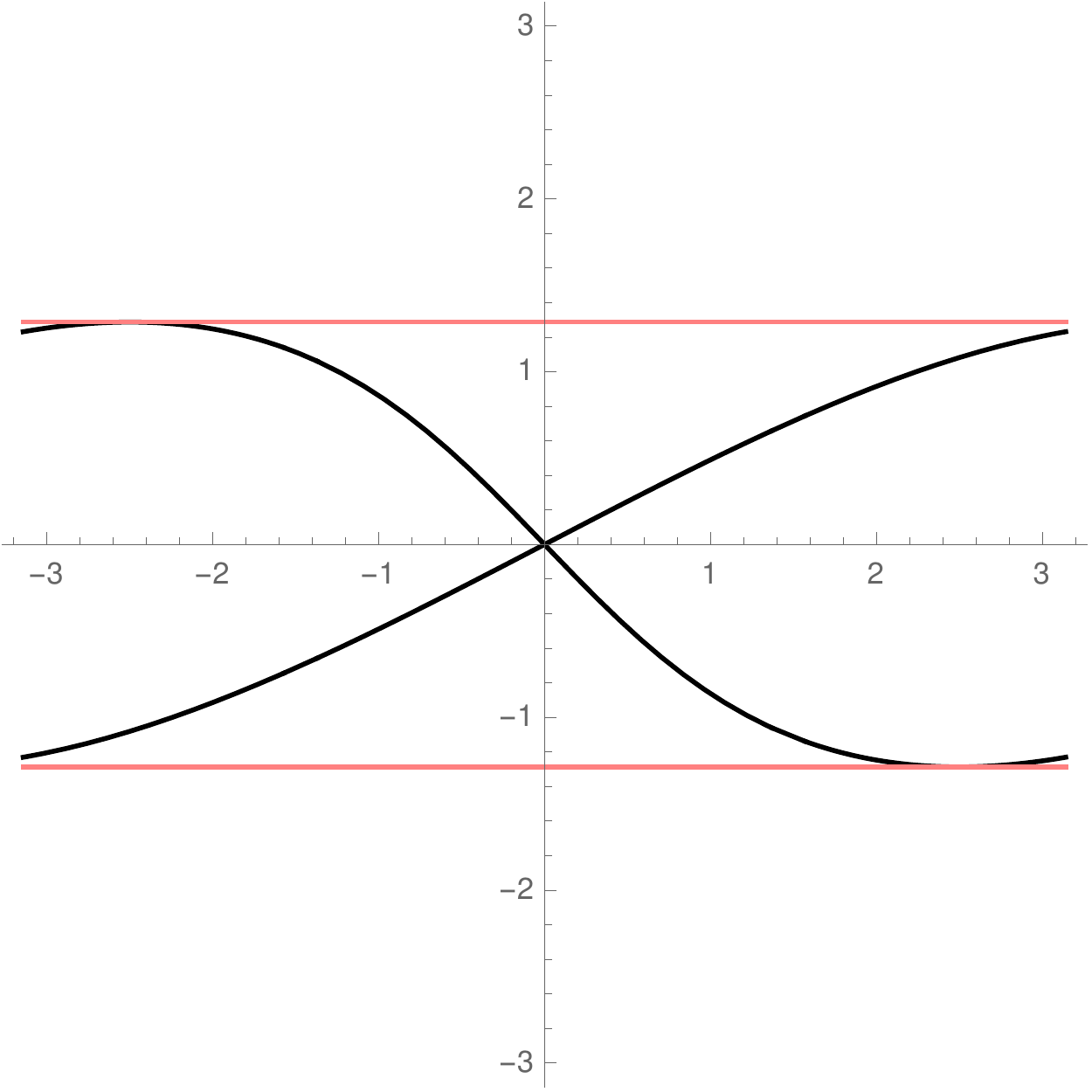}}
  \caption{\textsl{Iteration of $\Phi$ given by \eqref{AMYRIF} on $\mathbb{T}^2$. (Successive images of the vertical axis marked red)}}
  \label{AMYplots}
\end{figure}
Consider the rational inner function
\begin{equation}
\phi(z_1,z_2)=-\frac{\tilde{p}(z)}{p(z)}=-\frac{4z_1z_2^2 - z_2^2 - 3z_1z_2 - z_2 + z_1}{4 - z_1 - 3z_2 - z_1z_2 + z_2^2},
\label{AMYRIF}
\end{equation}
which features as an example in the paper \cite{AMY12}, and has been further studied in \cite[Section 1.3]{BPS20} and \cite[Example 5.2]{BCSprep} (without the minus sign in front). Since $\phi$ has a single singularity at $(1,1)$ and $\phi^*(1,1)=1$, the corresponding RISP $\Phi$ has a SF-point at $(1,1)$. We further check that $\phi(1,z_2)=1$, confirming the presence of a collapsing fiber.

In this example, $\alpha=\pi$, and
\[p(z)=p_1(z_2)+z_1p_2(z_2)\]
with
\[p_1(z_2)=4-3z_2+z_2^2\quad \textrm{and}\quad p_2(z_2)=-(1+z_2).\]
We have $p_2(-1)=0$, and as expected $\phi(z_1,-1)=-z_1$, a rotation of order $2$.
A short computation shows that the associated polynomial $Q_{\alpha}$ is given by
\begin{equation}
Q_{\alpha}(z_2)=25(z_2-1)^2\left(z_2^2-\frac{14}{25}z_2+1\right).
\label{AMYQ}
\end{equation}
Then $Q_{\alpha}(z_2)$ has a double root at $\lambda_1=1$, and two roots at $\lambda_{2,3}=\frac{1}{25}(7\pm 24i)$. 

Figure \ref{AMYplots} displays the dynamics of $\Phi$ on $\mathbb{T}^2$. We notice a single rotation belt bounded by $F_{\lambda_2}$ and $F_{\lambda_3}$, one fewer than the maximum given the degree of $Q_{\alpha}$. Also visible are two curves containing hyperbolic fixed points with a normal crossing at $(1,1)$, which is contained in the collapsing fiber.
\begin{figure}
\includegraphics[width=0.3 \textwidth]{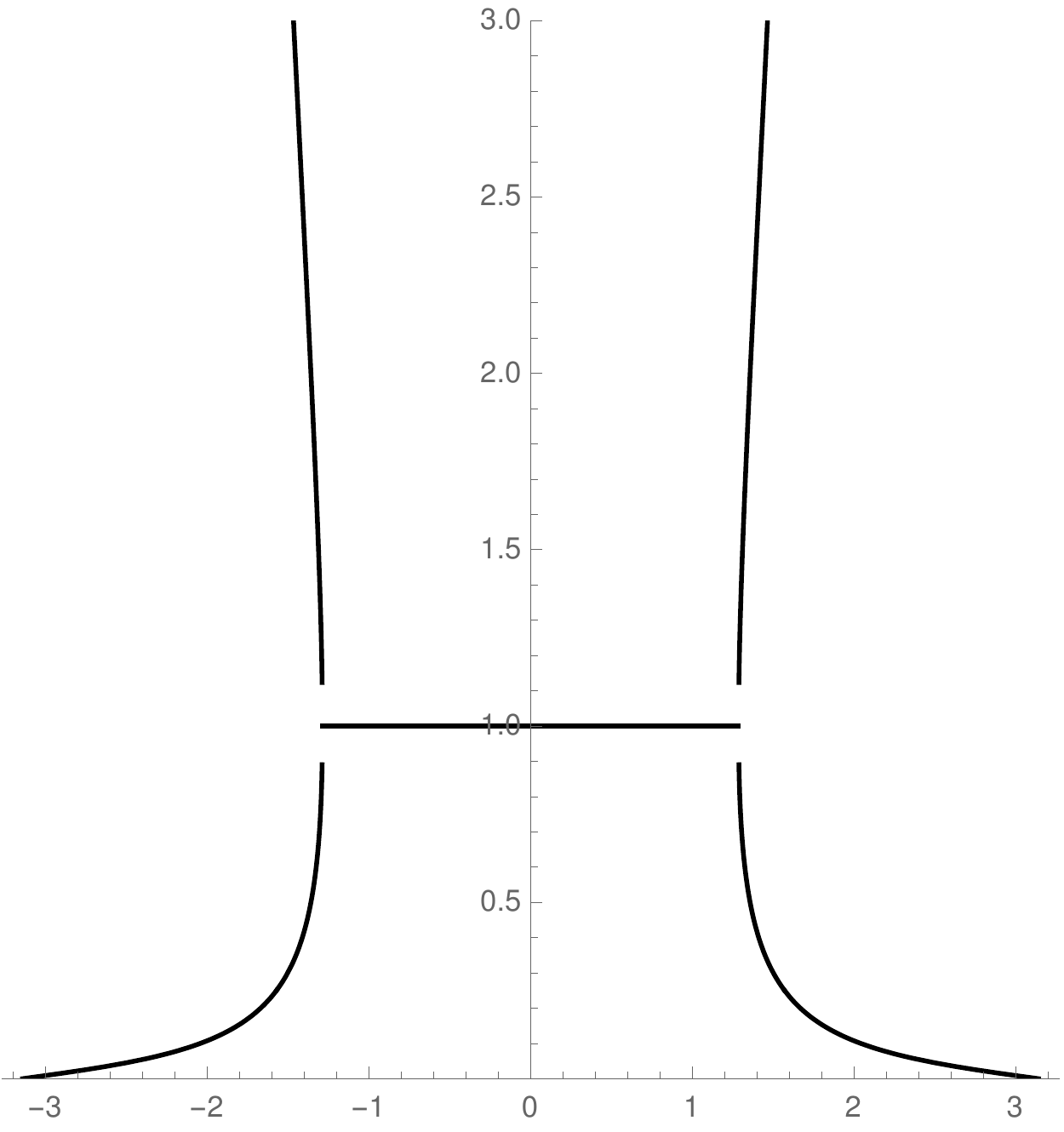}
\caption{Plot of $|\psi^{1,2}(e^{it_2})|$ showing interior/exterior fixed points coming together at parabolic fixed points.}
\label{AMYbranchpts}
\end{figure}
The hyperbolic fixed points are visible on a pair of curves parametrized by
\[\psi^{1,2}_{-1}(z_2)=\frac{1}{2(1+z_2)}\left(5 - 6 z_2 +5 z_2^2 \pm (-1 + z_2) \sqrt{25 - 14 z_2 + 25 z_2^2}\right).\]
These curves exhibit a normal crossing at $(1,1)$, the singularity of the RIF, reflecting the double root of $Q_{\alpha}$ at $(1,1)$. The branch point nature of the parabolic fixed points at $\lambda_2$ and $\lambda_3$ can be observed in Figure \ref{AMYbranchpts} which displays the absolute value of the functions $\psi^{1,2}(e^{it_2})$.

\end{ex}

\begin{ex}\label{ex:multisingex}
\begin{figure}[h!]
    \subfigure[$\Phi^n$ for $n=1$.]
      {\includegraphics[width=0.35 \textwidth]{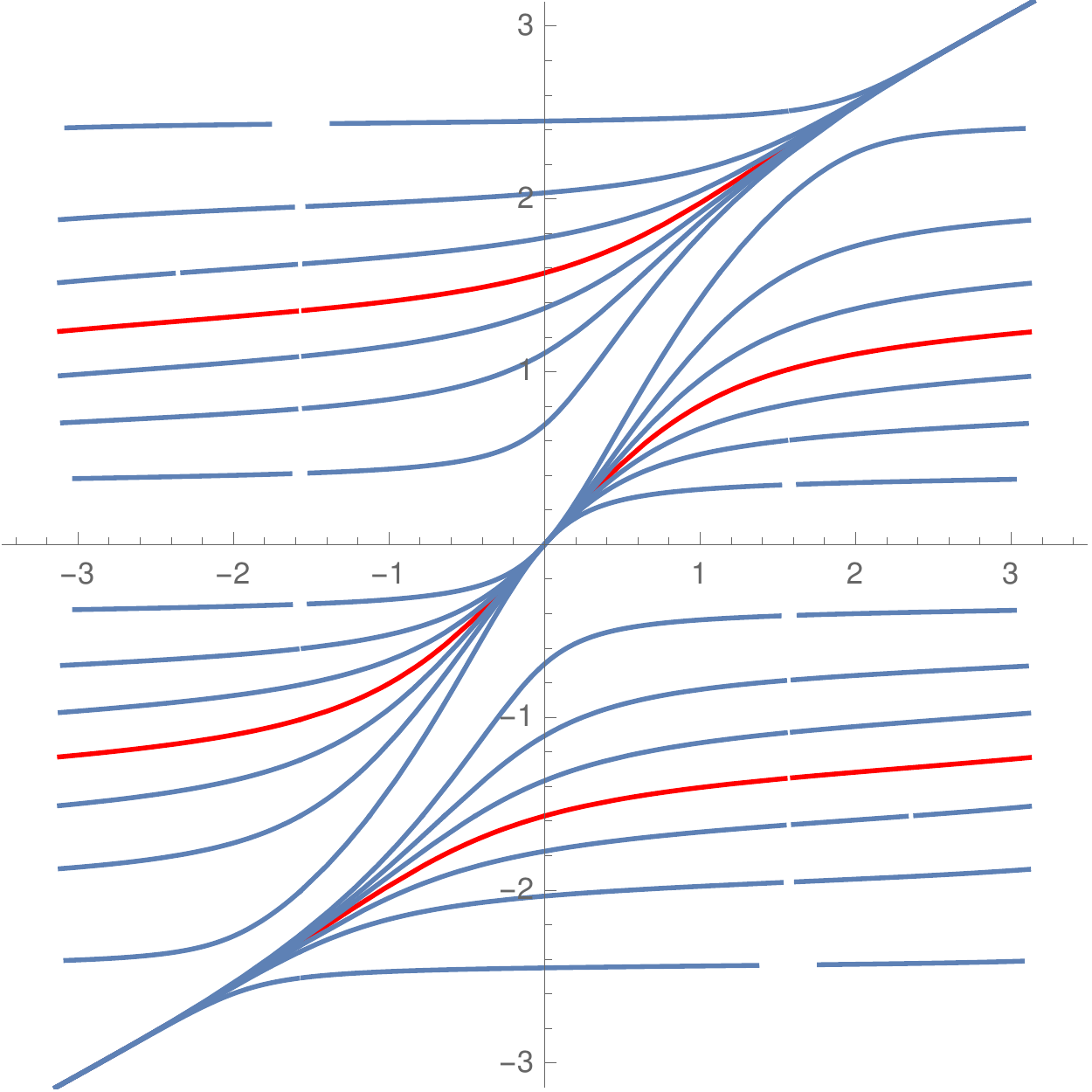}}
    \hfill
    \subfigure[$\Phi^n$ for $n=2$.]
      {\includegraphics[width=0.35 \textwidth]{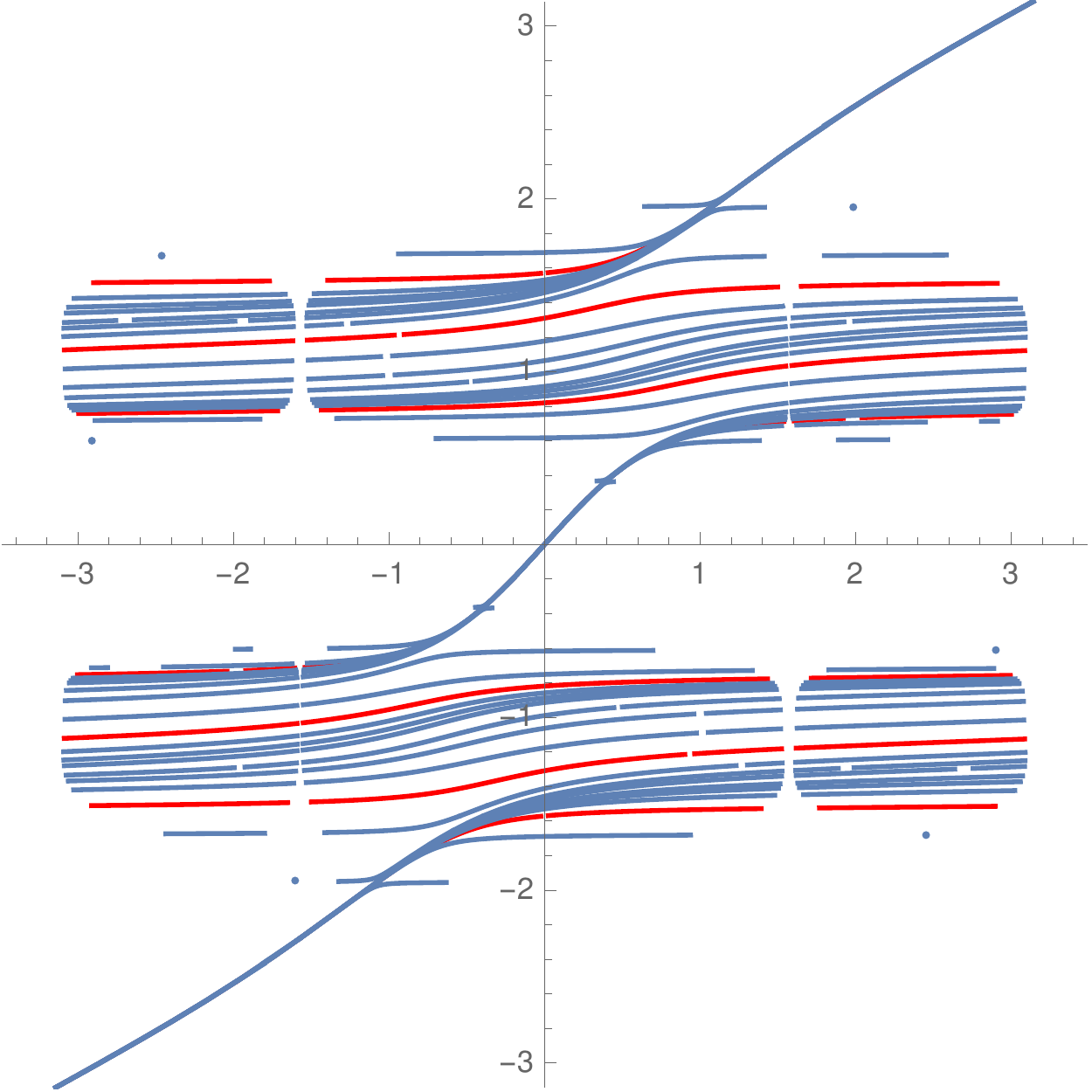}}
      \subfigure[$\Phi^n$ for $n=5$]
      {\includegraphics[width=0.35 \textwidth]{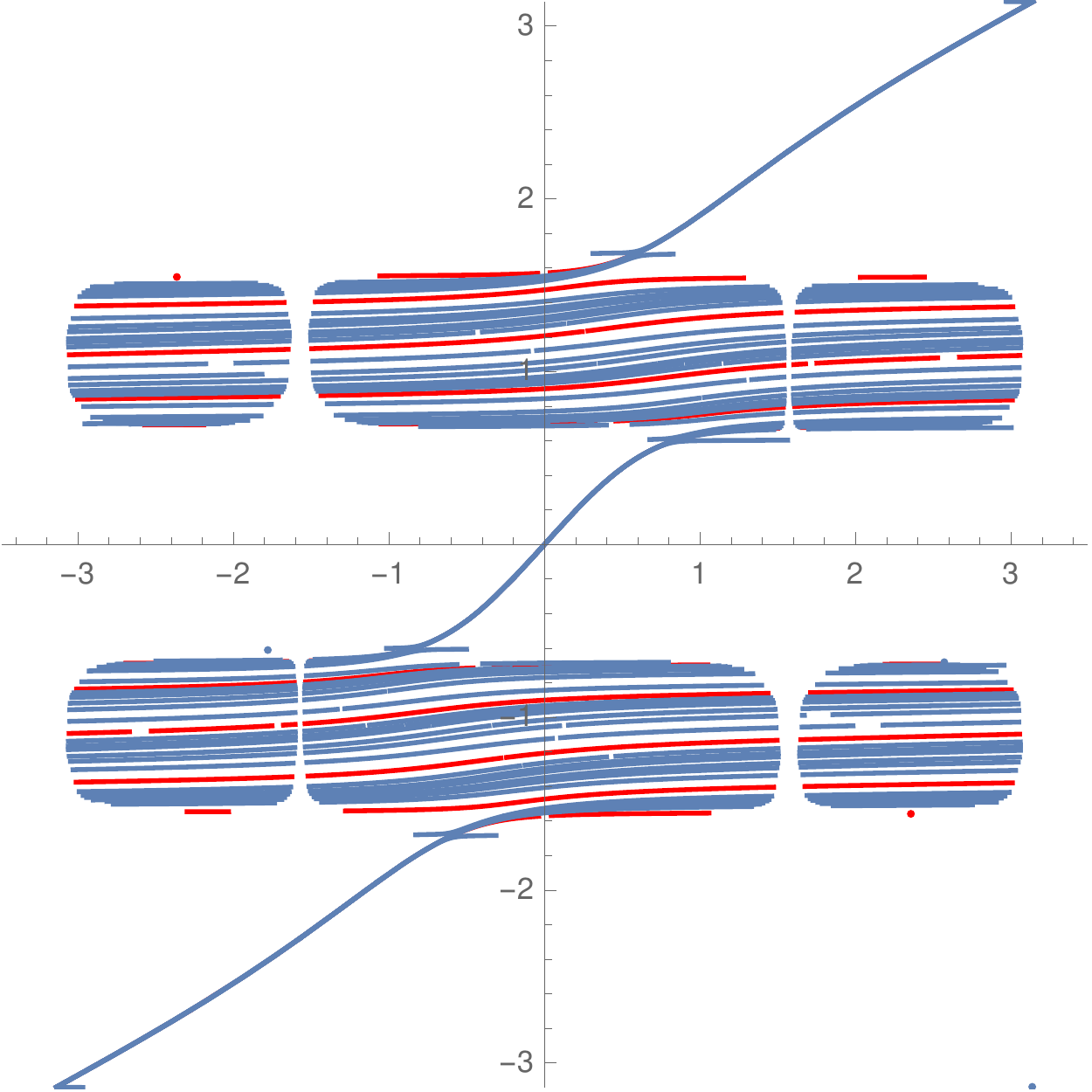}}
    \hfill
    \subfigure[$\Gamma(\Phi)\cap\mathbb{T}^2$. Parabolic fibers $F_{\lambda_{2,3}}$ in pink.]
      {\includegraphics[width=0.35 \textwidth]{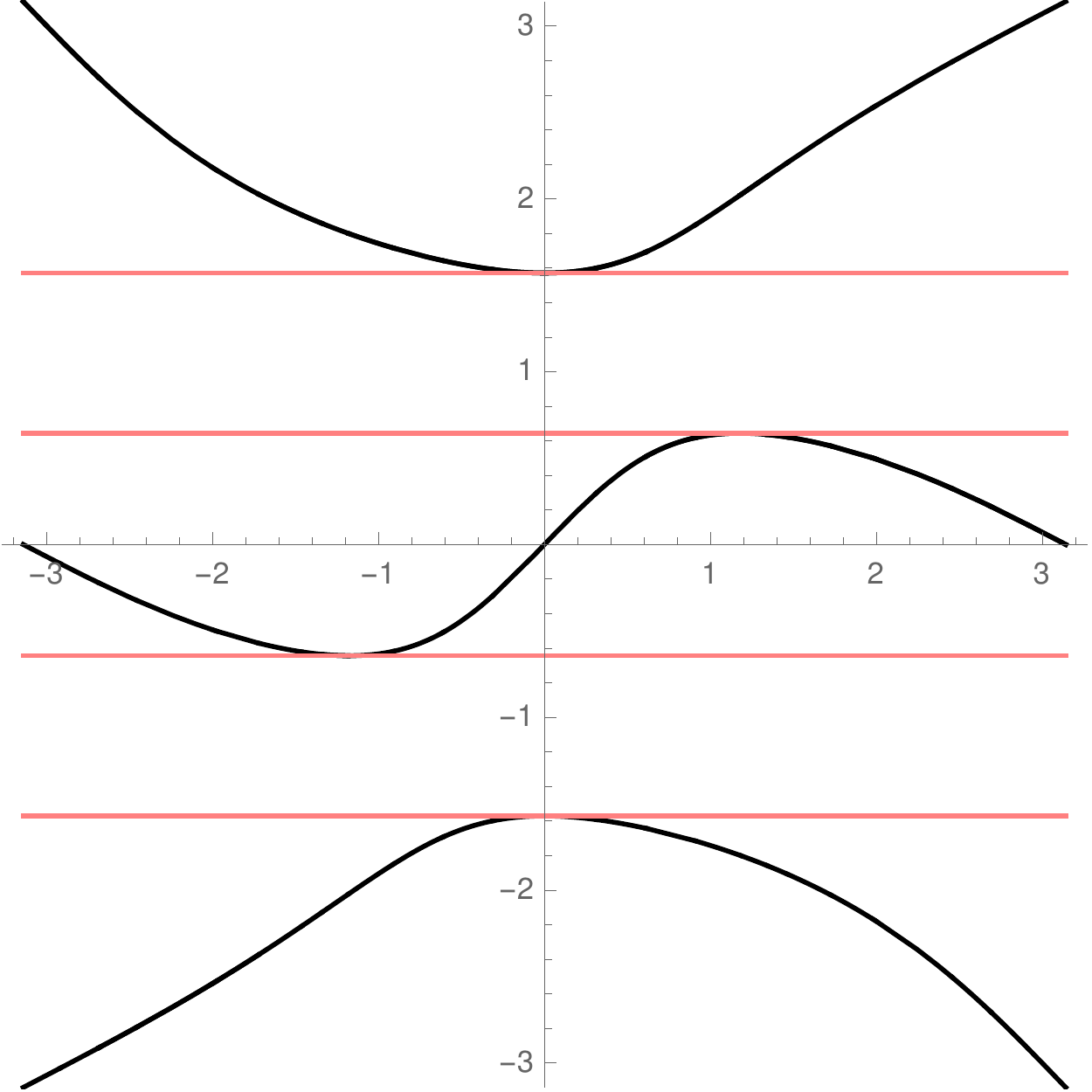}}
  \caption{\textsl{Iteration of $\Phi$ given by \eqref{multp} on $\mathbb{T}^2$. (Successive images of the vertical axis marked red)}}
  \label{MultSingplots}
\end{figure}
We turn to a RIF with multiple singularities on $\mathbb{T}^2$ and multiple rotation belts. Consider the RIF $\phi=-\frac{\tilde{p}}{p}$ with
\begin{equation}
p(z)=4 - z_1 + z_1z_2 - 3z_1z_2^2 - z_1z_2^3 \quad \textrm{and}\quad \tilde{p}(z)=4z_1z_2^3 - z_2^3 +z_2^2 - 3z_2 - 1,
\label{multp}
\end{equation}
a slightly modified form of  \cite[Example 7.4]{BPS20} (see also \cite[Example 5.4]{BCSprep}).

We have $p(1,1)=0$ and $p(-1,-1)=0$, and a computation shows that $\phi(z_1,1)=1$ and $\phi(z_1,-1)=-1$, confirming the presence of two collapsing fibers as guaranteed by Lemma \ref{lem:collapsefiberlem}, and so $\phi$ has two SF-points on the $2$-torus.  

We read off that
\[p_1(z_2)=4 \quad  \textrm{and}\quad p_2(z_2)=z_2^3+3z_2^2-z_2+1\]
and 
\[\tilde{p}_1(z_2)=4z_2^3 \quad \textrm{and} \quad \tilde{p}_2(z_2)=z_2^3-z_2^2+3z_2+1,\]
and after some simplifications (using that $\alpha=0$ here), we find that
\begin{equation}
Q_{\alpha}(z_2)=4(z_2+1)^2(z_2^2+1)(5z_2^2-8z_2+5).
\label{multQ}
\end{equation}
The polynomial $Q_{\alpha}$ has a double root at $z_2=-1$ (the $z_2$-coordinate of the SF-point at $(-1,-1)$), and simple roots at $z_2=\pm i$ and $z_2=\frac{1}{5}(4\pm3i)$. On the other hand $Q_{\alpha}(1)\neq 0$. This is reflected in the images in Figure \ref{MultSingplots}: $\Gamma(\Phi)$ has two components coming together at $(-1,-1)\in \Lambda^{\flat}$ with a normal crossing, while $(1,1)\in \Lambda^{\flat}$ is a generic point, in the sense that $\Gamma(\Phi)$ does not have a self-crossing. The simple zeros of $Q_{\alpha}$ correspond to parabolic points bounding two distinct rotation belts, and these in turn lie between two arrangements of curves with hyperbolic fixed points.
\end{ex} 

\section*{Acknowledgments}

We thank J.E. Pascoe and the noncommutative dynamics reading group organized by David Jekel for providing inspiration for this work.


\begin{thebibliography}{}

\end{thebibliography}


\begin{thebibliography}{alpha}
\bibliographystyle{apalike}

\bibitem{Aba98}M. Abate, The Julia-Wolff-Carath\'eodory theorem in polydisks, J. Anal. Math. {\bf 74} (1998), 275-306.

\bibitem{AMS06}J. Agler, J.E. M\McC Carthy, and M. Stankus, Toral algebraic sets and function theory on polydisks, J. Geom. Anal. {\bf 16} (2006), 551-562.

\bibitem{AMY12}J. Agler, J.E. M\McC Carthy, and N.J. Young, A Carath\'eodory theorem for the bidisk via Hilbert space methods, Math. Ann. {\bf 352} (2012), 581-624.

\bibitem{AG16}L. Arosio and P. Gumenyuk, Valiron and Abel equations for holomorphic self-maps of the polydisc, Internat. J. Math. {\bf 27} (2016), 1650034.

\bibitem{BearBook}A. Beardon, {\it Iteration of rational functions.} Graduate Texts in Mathematics 132, Springer-Verlag, New York. 1991.

\bibitem{BCSprep}K. Bickel, J.A. Cima, and A.A. Sola, Clark measures for rational inner functions, preprint, available at https://arxiv.org/abs/2101.00508

\bibitem{BKPSprep}K. Bickel, G. Knese, J.E. Pascoe, and A. Sola, Local theory of stable polynomials and bounded rational functions of several variables, preprint 2021.

\bibitem{BPS18}K. Bickel, J.E. Pascoe, and A. Sola, Derivatives of rational inner functions: geometry of singularities and integrability at the boundary, Proc. London Math. Soc. \textbf{116} (2018), 281-329.

\bibitem{BPS20}K. Bickel, J. E. Pascoe, and A. Sola, Level curve portraits of rational inner functions, Ann. Sc. Norm. Sup. Pisa Cl. Sc. {\bf XXI} (2020), 451-494.

\bibitem{BPSajm}K. Bickel, J. E. Pascoe, and A. Sola, Singularities of rational inner functions in higher dimensions, Amer. J. Math., to appear.

\bibitem{CGBook}L. Carleson and T.W.  Gamelin, {\it Complex dynamics}, Universitext: Tracts in Mathematics, Springer-Verlag, New York, 1993.

\bibitem{C32}H. Cartan, Les fonctions de deux variables complexes. L'it\'eration des transformations int\'erieures d'un domaine born\'e (French), Math. Z. {\bf 35} (1932), 760-773.

\bibitem{F03}C. Favre, Les applications monomiales en deux dimensions, Michigan Math. J. {\bf 51} (2003), 467-475.

\bibitem{ForSurv}J.E. Fornaess, {\it Dynamics in several complex variables}. CBMS Regional Conference Series in Mathematics, 87. Providence, RI, 1996.

\bibitem{F05}C. Frosini, Dynamics on bounded domains, in: {\it The $p$-harmonic equation and recent advances in analysis}, 99-117, Contemp. Math. {\bf 370} Amer. Math. Soc., Providence, RI, 2005.

\bibitem{H54}M. Herv\'e, Sur l'it\'eration des transformations analytiques dans le bicercle unit\'e, Ann. Sci. Ecole Norm. Sup. (3) {\bf 71} (1954), 1-28.

\bibitem{J99}M. Jonsson, Dynamics of polynomials skew-products on $\mathbb{C}^2$, Math. Ann. {\bf 314} (1999), 403-447.

\bibitem{Kne10}G. Knese, Polynomials defining distinguished varieties, Trans. Amer. Math. Soc. {\bf 362} (2010), 5635-5655.

\bibitem{Kne15}G. Knese, Integrability and regularity of rational functions, Proc. London. Math. Soc. {\bf 111} (2015), 1261-1306.

\bibitem{MilBook}J. Milnor, {\it Dynamics in one complex variable}. Annals of Mathematics Studies 160, Princeton Univ. Press, Princeton, NJ, 2006. 

\bibitem{NThesis}J. Nowell, Denjoy-Wolff sets for analytic maps on the polydisk, PhD Thesis, University of Florida, 2019.

\bibitem{PR19}H. Peters and J. Raissy, Fatou components of elliptic polynomial skew-products, Ergodic Theory Dynam. Systems {\bf 39} (2019), 2235-2247.

\bibitem{PR10}E.R. Pujals and R.K.W. Roeder, Two-dimensional Blaschke products: degree growth and ergodic consequences, Indiana Univ. Math. J. {\bf 59} (2010), 301-325.

\bibitem{PS08}E.R. Pujals and M. Shub, Dynamics of two-dimensional Blaschke products, Ergodic Theory Dynam. Systems {\bf 28} (2008), 575-585.

\bibitem{Rud69} W. Rudin, \emph{Function Theory in polydisks},
W. A. Benjamin, Inc., New York-Amsterdam, 1969.

\bibitem{SibSurv}N. Sibony, Dynamique des applications rationelles de $\mathbf{P}^k$. (French) {\it Dynamique et g\'eom\'etrie complexes (Lyon, 1997)}, Panor. Synth\`eses {\bf 8}, Soc. Math. France, Paris, 1999.

\bibitem{SBJ17}J. Slipantschuk, O.F. Bandtlow, and W. Just, Complete spectral data for analytic Anosov maps of the torus, Nonlinearity {\bf 30} (2017), 2667-2686.

\bibitem{TD17}R. Tully-Doyle, Analytic functions on the bidisk at boundary singularities via Hilbert space methods, Oper. Matrices {\bf 11} (2017), 55-70.

 \bibitem{TDpage}R. Tully-Doyle, Personal webpage \url{https://rtullydo.github.io/RISP-dynamics}.

\end{thebibliography}
\end{document}